\documentclass[a4paper, oneside, reqno, 10pt, final]{amsart}
\usepackage{ifxetex}

\ifxetex
  \usepackage{fontspec}
  
  \usepackage[libertine]{newtxmath}
\else
  \usepackage[utf8]{inputenc}
  \usepackage[T1]{fontenc}
  \usepackage[osf]{mathpazo}
\fi

\usepackage[style=alphabetic, maxcitenames=4, maxbibnames=10, backend=bibtex]{biblatex}
\DeclareFieldFormat{postnote}{#1}

\usepackage{microtype}
\usepackage[a4paper]{geometry}
\usepackage{tikz}
\usetikzlibrary{matrix,arrows}

\usepackage{amsmath, amsthm, amssymb}
\usepackage[all, pdftex]{xy}
\usepackage{mathtools}

\usepackage{hyperref}
\hypersetup{ 
    pdftitle = {The Galois action on M-Origamis and their Teichmüller curves},
    pdfauthor = {Florian Nisbach}
}

\newcommand{\OriSquare}[5]
{
\begin{xy}
(0,0)="Pos";
"Pos"+(1,0) **@{-}; ?*h!U!/^2pt/{\text{\tiny #5}}, 
"Pos"+(1,1) **@{-}; ?*h!L!/^1pt/{\text{\tiny #2}},
"Pos"+(0,1) **@{-}; ?*h!D!/^1pt/{\text{\tiny #3}},
"Pos" **@{-}; ?*h!R!/^1pt/{\text{\tiny #4}},
"Pos"+(0.5,0.5) *h{#1};
\end{xy}
}

\newcommand{\vmathbb}{\mathbb}
\DeclareMathOperator{\Gal}{Gal}
\DeclareMathOperator{\QQ}{\vmathbb{Q}}
\DeclareMathOperator{\QQq}{\overline{\QQ}}
\DeclareMathOperator{\absGal}{\Gal(\QQq/\QQ)}
\DeclareMathOperator{\CC}{\vmathbb{C}}
\DeclareMathOperator{\RR}{\vmathbb{R}}
\DeclareMathOperator{\ZZ}{\vmathbb{Z}}
\DeclareMathOperator{\NN}{\vmathbb{N}}
\DeclareMathOperator{\Af}{\vmathbb{A}}
\DeclareMathOperator{\Hp}{\vmathbb{H}}
\DeclareMathOperator{\Pe}{\vmathbb{P}^1}
\DeclareMathOperator{\Pek}{\vmathbb{P}^1_k}
\DeclareMathOperator{\PeC}{\vmathbb{P}^1_{\CC}}
\DeclareMathOperator{\crit}{crit}
\DeclareMathOperator{\Stab}{Stab}
\DeclareMathOperator{\Aut}{Aut}
\DeclareMathOperator{\Out}{Out}
\DeclareMathOperator{\Deck}{Deck}
\DeclareMathOperator{\id}{id}

\DeclareMathOperator{\cupdot}{\stackrel{.}{\cup}}
\DeclareMathOperator{\Spec}{Spec}
\DeclareMathOperator{\im}{im}

\DeclareMathOperator{\SL}{SL}
\DeclareMathOperator{\GL}{GL}
\DeclareMathOperator{\PSL}{PSL}

\DeclareMathOperator{\len}{len}

\newenvironment{aenum}%
	{\begin{enumerate}  \setlength{\itemsep}{0em}}%
	{\end{enumerate}}

\theoremstyle{definition}
\newtheorem{defbem}{Definition and Remark}[section]
\newtheorem{defi}[defbem]{Definition}
\newtheorem{bem}[defbem]{Remark}
\newtheorem{prop}[defbem]{Proposition}
\newtheorem{kor}[defbem]{Corollary}
\newtheorem{hilfssatz}[defbem]{Lemma}
\newtheorem{satz}{Theorem}

\newtheorem{bsp}[defbem]{Example}

\title{}

\begin{document}

\begin{center}
  {\large\bf \uppercase{The Galois action on M-Origamis\\ and their Teichmüller curves}}\\[1em]  
  {\uppercase{Florian Nisbach}\\[1em]}
  Karlsruhe Institute of Technology (KIT)\\
  {florian.nisbach@kit.edu\\[1em]}
  \textit{\today}
\end{center}

\begin{abstract}
We consider a rather special class of translation surfaces (called \emph{M-Origamis} in this work) that are obtained from dessins by a construction introduced by Möller in \cite{mm1}. We give a new proof with a more combinatorial flavour of Möller's theorem that $\absGal$ acts faithfully on the Teichmüller curves of M-Origamis and extend his result by investigating the Galois action in greater detail. 

We determine the Strebel directions and corresponding cylinder decompositions of an M-Origami, as well as its  Veech group, which contains the modular group $\Gamma(2)$ and is closely connected to a certain group of symmetries of the underlying dessin. Finally, our calculations allow us to give explicit examples of Galois orbits of M-Origamis and their Teichmüller curves.
\end{abstract}

\maketitle

\tableofcontents

\section{Introduction} 
The absolute Galois group $\absGal$ has been a central object of interest for quite some time. Its appeal is the vast amount of arithmetic information it encodes, which is also an explanation for its tremendously complicated structure. To give one example, the question which isomorphism types of groups appear as finite quotients of $\absGal$ is the inverse Galois problem, a still prospering field of research with wide open problems. A classical review on known results and open questions about $\absGal$ is \cite{neu}.

One approach to understanding the absolute Galois group is to study its actions on objects that are relatively easy to understand. Belyi's theorem \cite{be} (see Section \ref{sec:dessins}) inspired Grothendieck to define a class of such combinatorial objects, the so-called \emph{dessins d'enfants}, on which the absolute Galois group acts faithfully. One way of describing a dessin d'enfant is seeing it as a covering of the complex projective line ramified over three points. In this sense, a related class of objects are \emph{origamis} (sometimes also called \emph{square tiled surfaces}), which can be seen as coverings of an elliptic curve ramified over a single point. They allow an $\SL_2(\RR)$ action which gives rise to constructing so-called \emph{Teichmüller curves} in the corresponding moduli space. These curves turn out to be defined over number fields, so they also carry a Galois action. For some time it was unclear if this action is non-trivial, until Möller proved in \cite{mm1} that it is indeed faithful. To do this, he considered a subset of the set of dessins on which $\absGal$ still acts faithful, made origamis out of them, considered their Teichmüller curves and showed that the Galois actions on all these objects fit together in such a way that the faithfulness does not break along the way.

The goal of this work is to give a more topological or combinatorial view (in terms of the monodromy of coverings) on Möller's construction, which enables us to extend his results and actually give examples of Galois orbits of Teichmüller curves. The structure of this work is as follows:

In Section \ref{sec:topology}, we give a short overview of the topological methods we are going to use and adapt them to our needs. Section \ref{sec:dessins} is an introduction to Belyi theory and dessins d'enfants.

In Section \ref{sec:origamis}, we begin by explaining the analytical and algebro-geometric notions of Teichmüller curves of translation surfaces in general, and of origamis  in particular. We discuss the Galois action on these objects and prove the Galois invariance of certain properties of an origami $O$, such as the index of the Veech group $\Gamma(O)$  in $\SL_2(\ZZ)$ and the isomorphism type of its group of translations $\text{Trans}(O)$. Also, we obtain the maybe surprising relationship
\[[M(O):M(C(O))]\leq[\SL_2(\ZZ):\Gamma(O)],\]
where $M(O)$ and $M(C(O))$ are the fields of moduli of the origami $O$ and its Teichmüller curve, respectively.

The main part of this work is Section \ref{sec:m-ori}, where we begin by explaining Möller's fibre product construction to produce special origamis (which we call \emph{M-Origamis}) out of dessins. We show that we can take this fibre product in the category of topological coverings, which allows us to explicitly express the monodromy of an M-Origami $O_\beta$ in terms of the monodromy of the dessin $\beta$ we start with.
We show that the Veech group of $O_\beta$ lies between the full modular group $\SL_2(\ZZ)$ and $\Gamma(2)$ and exhibit its close relationship to the group $W_\beta\coloneqq \text{Stab}_{\Aut(\PeC\setminus\{0,1,\infty\})}(\beta)$. We go on by determining the cylinder decomposition of $O_\beta$ in terms of the ramification type of $\beta$. These calculations will allow us to reprove Möller's results and refine them, e.g.\ by showing that if for a \emph{Belyi tree} $\beta$ and an automorphism $\sigma\in\absGal$, we have $\beta\ncong \beta^\sigma$, then we also have $O_\beta\ncong O_{\beta^\sigma}=(O_\beta)^\sigma$, and under a certain condition on the field of moduli of $\beta$, their Teichmüller curves are also different.

Finally, the last section is dedicated to giving several examples that we are able to construct with our methods. We explicitly construct Galois orbits of M-Origamis and their Teichmüller curves. Also, we show that every congruence  subgroup of $\SL_2(\ZZ)$ of level $2$ can be realised as the Veech group of an origami. This is interesting in the light of  \cite[Theorem 4]{gs}, where Weitze-Schmithüsen realises all congruence subgroups of $\SL_2(\ZZ)$ as Veech groups of origamis, with the exception of a finite list containing the congruence group of level $2$ and index $2$. 

For more details on a great part of this article's contents, see also the author's PhD thesis \cite{diss}.

\textbf{Acknowledgements.}\: The author wishes to express his thanks to the advisors of his thesis, Gabriela Weitze-Schmithüsen and Frank Herrlich, as well as to Stefan Kühnlein, for many helpful discussions and suggestions. Furthermore, many thanks go to the authors of the origami software package \cite{ori-lib} that was used to produce the examples in the last section of this work.

\section{Topological preliminaries}\label{sec:topology}
In this section we will recall some basic properties of topological coverings. In particular we will write up formulas for the monodromy of the fibre product and the composition of two coverings in terms of their monodromies, as these seem to be absent in most topology textbooks. Let us start with some definitions and conventions.

Throughout this article, a (topological, unramified) \textit{covering} is understood to be a continous map $p\colon Y\to X$, where $X$ is a path-wise connected, locally path-wise connected, semi-locally simply connected topological space (call these spaces \textit{coverable}), and every point $x\in X$ has a neighbourhood $U_x\ni x$ such that $p^{-1}(U_x)$ is a disjoint union $\coprod_{i\in I}U_i$ such that for all $i\in I$ the restriction $p_{|U_i}:U_i\to U_x$ is a homeomorphism. $U_x$ is then called \textit{admissible neighbourhood} of $x$ with respect to $p$. Sometimes we will lazily drop the specific covering map $p$ and simply write $Y/X$. The well defined cardinality $\deg p\coloneqq|p^{-1}(x)|$ is called \textit{degree} of the covering.

Note that usually, the definition of a covering requires also $Y$ to be path-wise connected. For this situation we will use the term \emph{connected covering}.

As usual, denote the push-forward of paths (or their homotopy classes) by a continous map $f$ by $f_\ast$, i.e.\ $f_\ast(g)\coloneqq f\circ g$. For a continous map of pointed spaces $f\colon (Y,y)\to (X,x)$, this yields a (functorial) group homomorphism $f_\ast\colon \pi_1(Y,y) \to \pi_1(X,x)$, which is injective if $f$ is a covering. In case $f$ is a connected covering  and $f_\ast(\pi_1(Y,y))\subseteq \pi_1(X,x)$ is a normal subgroup, $f$ is called normal or Galois covering. Remember that in this case the factor group is isomorphic to $\text{Deck}(Y/X)$, the group of deck transformations for the covering, i.e.\ the homeomorphisms of $Y$ preserving the fibres of $f$, which is then often called the \emph{Galois group} of $f$.

Recall the well-known path lifting property of coverings: Let $g\colon [0,1]\to X$ be path and $p\colon Y \to X$ a covering, then for every $y\in p^{-1}(g(0))$ there exists a unique lift  of $g$ with $h(0)=y$. Denote this lift by $L_y^p(g)$ and its endpoint by $e_y^p(g)\coloneqq h(1)$.
We compose paths “from right to left”, more precisely: If $\alpha, \beta\in\pi_1(X,x_0)$ are two elements of the fundamental group of a topological space then $\beta\alpha$ shall denote the homotopy class one gets by first passing through a representative of $\alpha$ and then one of $\beta$.

The monodromy of a covering $p\colon Y \to X$ is defined as follows: Fix a basepoint $x_0\in X$ and a numbering $p^{-1}(x_0)=\{y_1,\ldots,y_d\}$ on its fibre. Then the monodromy homomorphism $m_p$ is given by
\[m_p\colon\pi_1(X,\,x_0)\to S_d, \gamma\mapsto (i\mapsto j,\text{ if }y_j=e^p_{y_i}(\gamma)).\]
Note that  $p$ is a connected covering (i.e.\ $Y$ is path-wise connected) iff the image of $m_p$ is a transitive subgroup of $S_d$. Of course, as we drop the requirement of $Y$ being path-wise connected, the well-known Galois correspondence between equivalence classes of coverings and conjugacy classes of subgroups of the fundamental group breaks. Instead, we have the following easy
\begin{prop}\label{prop:coverings} Let $X$ be a coverable space and $n\in\NN$. Then there is a bijection
\begin{align*}
 & \{ X'/X \text{ covering of degree } d \}_{/\text{Fibre preserving homeomorphisms}}\\
\leftrightarrow & \{m\colon \pi_1(X) \to S_d \text{ permutation representation }\}_{/\text{conjugation in } S_d}.
\end{align*}
\end{prop}
Now let us turn to fibre products of coverings. Recall that for two continous maps $f\colon A\to X,\, g\colon B \to X$, the \textit{fibre product} $A\times_X B$ is given by
\[A\times_X B\coloneqq \left\{(a,\,b)\in A\times B\colon f(a)=g(b)\right\}\]
endowed with the subspace topology of the product. Consequently, the projections $p_A\colon A\times_X B\to A, (a,b)\mapsto a$ and $p_B\colon A\times_X B\to B, (a,b)\mapsto b$ are continuous. It is easy to see that if $f$ and $g$ are covering maps, so is $f\circ p_A=g\circ p_B\colon A\times_X B\to X$, and so are $p_A$ and $p_B$ if $A$ and $B$ are path-wise connected, respectively.

Note that, even if both $A$ and $B$ are path-wise connected, the fibre product $A\times_X B$ need not be. In fact, the following proposition will show that, if $f\colon A\to X$ is a degree $d$ covering, then $A\times_X A$ is the coproduct of $d$ copies of $A$.

From the proposition above we know that a covering is uniquely determined by its monodromy. So one should be able to write down a formula for the monodromy of the fibre product of two coverings in terms of their respective monodromies. The following proposition gives an answer.

\begin{prop}\label{satz:fprod}
Let $X$ be a coverable space, $f\colon A\to X,\, g\colon B\to X$ coverings of degree $d$ and $d'$, respectively, with given monodromy maps $m_f$ resp. $m_g$. Then, we have for the fibre product $A\times_X B$:
\begin{aenum}
\item For each path-wise connected component $A_i\subseteq A$, the restriction 
\[{p_A}_{|p_A^{-1}(A_i)}\colon p_A^{-1}(A_i)=A_i\times_X B \to A_i\]
is a covering of degree $d'$ with monodromy
\[m_g\circ(f_{|A_i})_\ast.\]
\item The map $f\circ p_A=g\circ p_B \colon A\times_X B\to X$ is a covering of degree $dd'$ with monodromy
\[m_f\times m_g\colon  \pi_1(X,\,x_0)\to S_d\times S_{d'}\subseteq S_{dd'}, \gamma\mapsto \left((k,\,l)\mapsto (m_f(k),\,m_g(l))\right),\]
where $(k,\,l)\in\{1,\ldots,\,d\}\times\{1,\ldots,\,d'\}$.
\end{aenum}
\end{prop}

\begin{proof} We omit the straightforward proof that ${p_A}_{|p_A^{-1}(A_i)}$ and $f\circ p_A=g\circ p_B$ are coverings of the claimed degrees. For part a), let $A$ w.l.o.g.\ be path-wise connected, i.e.\ $A_i=A$.
Let us now calculate the monodromy of $p_A$. So, choose base points $x_0\in X$ and $a_0\in f^{-1}(x_0)$. Let $g^{-1}(x_0)=\{b_1,\ldots,\,b_{d'}\}$ be the fibre over $x_0$, and $p_A^{-1}(a_0)=\{c_1,\ldots,\,c_{d'}\}$ be the fibre over $a_0$, the numbering on the latter chosen such that $p_B(c_i)=b_i$.

Now, take a closed path $\gamma\colon [0,\,1]\to A$ with $\gamma(0)=\gamma(1)=a_0$. Let $i\in\{1,\ldots,\,d'\}$, and let $\tilde{\gamma}=L_{c_i}^{p_A}(\gamma)$ be the lift of $\gamma$ starting in $c_i$. Assume that $\tilde{\gamma}(1)=c_j$.

Consider now the path $\delta=f\circ\gamma$. It is a closed loop starting in $x_0$. Let $\tilde{\delta}=L_{b_i}^g(\delta)$ be its lift starting in $b_i$, then we have, because of the uniqueness of the lift and the commutativity of the diagram: $\tilde{\delta}=p_B\circ\tilde{\gamma}$, so particularly, as we asserted $p_B(c_i)=b_i$ for all $i$, we have $\tilde{\delta}(1)=b_j$.

So indeed, we have shown $m_{p_A}(i)=(m_g\circ f_\ast)(i)$.

For part b), let $\gamma \in \pi_1(X,x_0),\, c_{ij}\in (f\circ p_A)^{-1}(x_0)$. Further let $e_{c_{ij}}^{f\circ p_A}(\gamma)=c_{kl}$.
Then we have $e_{a_i}^f(\gamma)=p_A(c_{kl})=a_k$ and $e_{b_j}^g(\gamma)=p_B(c_{kl})=b_l$. This completes the proof.
\end{proof}

Now, let $f\colon Y\to X$ and $g\colon Z\to Y$ be coverings of degrees $d$ and $d'$, respectively (so in particular $Y$ is path-wise connected). It is straightforward to see that $f\circ g$ is a covering of degree $dd'$. Like above, we are interested in a formula for $m_{f\circ g}$ in terms of $m_f$ and $m_g$.

Let $x_0\in X,\, f^{-1}(x_0)=\{y_1,\ldots,y_d\},\, g^{-1}(y_i)=\{z_{i1},\ldots, z_{id'}\}$. The fundamental group of $X$ is denoted by $\Gamma\coloneqq \pi_1(X,x_0)$, the given mono\-dromy map by $m_f\colon \Gamma\to S_d$. Fix the notation 
\[\Gamma_1\coloneqq m_f^{-1}(\Stab(1))=\{\gamma\in\Gamma\colon  m_f(\gamma)(1)=1\}.\]
So, if we choose $y_1$ as a base point of $Y$ and set $\Gamma'\coloneqq \pi_1(Y,y_1)$, then we have $f_\ast(\Gamma')=\Gamma_1.$ Denote, as usual, the monodromy map of the covering $g$ by $m_g\colon \Gamma'\to S_{d'}$.

\begin{prop}\label{satz:comp}
In the situation described above, let $\gamma_i,\, i=1,\ldots,d$, be right coset representatives of  $\Gamma_1$ in $\Gamma$, with $\gamma_1=1$, such that $e_{y_i}^f(\gamma_i)=y_1$. So, we have $\Gamma=\:\stackrel{\cdot}{\bigcup}\Gamma_1\cdot\gamma_i$.

Then, we have:
\[m_{f\circ g}(\gamma)(i,\, j)=\left(m_f(\gamma)(i), m_g(c_i(\gamma))(j)\right)\]
Here, we denote $c_i(\gamma)\coloneqq (f_\ast)^{-1}(\gamma_k\gamma\gamma_i^{-1}),\, k\coloneqq m_f(\gamma)(i)$
\end{prop}

\begin{proof}
Let $\gamma\in\pi_1(X, x_0)$, $\alpha\coloneqq L_{z_{ij}}^{f\circ g}(\gamma)$, and further let $\alpha(1)=e_{z_{ij}}^{f\circ g}(\gamma)\eqqcolon z_{i'j'}$. We have to determine $(i',\,j')$.

The path $\beta\coloneqq L_{y_i}^f(\gamma)=g\circ\alpha$ has endpoint $\beta(1)=y_{m_f(\gamma)(i)}$. In particular, we have $i'=m_f(\gamma)(i)$.

Now,  let us determine $j'$. So let $\beta_\nu\coloneqq L_{y_\nu}^f(\gamma_\nu)$ be liftings (for $\nu=1,\ldots,\, d$). Remember that by our choice of the numbering of the $\gamma_\nu$, we have $\beta_\nu(1)=e_{y_\nu}^f(\gamma_\nu)=y_1$. Using the notation $k\coloneqq m_f(\gamma)(i)$, we can write $\beta=\beta_k^{-1}\tilde{\beta}\beta_i$ with unique $\tilde{\beta}\in\pi_1(Y, y_1)$. Indeed, we have: $\tilde{\beta}=\beta_k\beta\beta_i^{-1}=L_{y_1}^f(\gamma_k\gamma\gamma_i^{-1})$.

W.l.o.g. we have that the lifting $\alpha_{ij}\coloneqq L_{z_{ij}}^g(\beta_i)$ has endpoint $z_{1j}$, as we have chosen $\gamma_1=1$, and for $i\neq 1$ we can renumber the $z_{ij}, j=1,\ldots,d'$.

Denote $\tilde{\alpha}\coloneqq L_{z_{1j}}^g(\tilde{\beta})$ and $l\coloneqq m_g(\tilde{\beta}(j))$, then we have $\alpha=\alpha_{kl}^{-1}\tilde{\alpha}\alpha_{ij}$, and because of $\tilde{\alpha}(1)=z_{1\, m_g(\tilde{\beta}(j))}$ we get:
\[z_{i'j'}=e_{z_{ij}}^{f\circ g}(\gamma)=\alpha(1)=\left(\alpha_{kl}^{-1}\tilde{\alpha}\alpha_{ij}\right)(1)=z_{kl}\]
So finally, $j'=l=m_g(\tilde{\beta})(j)=m_g(\beta_k\beta\beta_i^{-1})(j)=m_g(c_i(\gamma))(j)$.
\end{proof}

Before we move on, let us state a lemma on normal coverings, which will become handy later on and which the author has learned from Stefan Kühnlein.

\begin{hilfssatz}\label{hilfssatz:decktrafo} Let $f\colon X\to Y,\,f'\colon X'\to Y$ be connected coverings and $g\colon Y\to Z$ be a normal (so in particular connected) covering, and let $Z$ be a Hausdorff space. If $g\circ f\cong g\circ f'$, i.e.\ there is a homeomorphism $\varphi\colon X\to X'$ with $g\circ f=g\circ f'\circ \varphi$, then there is a deck transformation $\psi\in \Deck(g)$ such that $\psi\circ f=f'\circ \varphi$.
\end{hilfssatz}

\begin{center}
\begin{tikzpicture}[description/.style={fill=white,inner sep=2pt}]
\matrix (m) [matrix of math nodes, row sep=1.5em,
column sep=1.25em, text height=1.5ex, text depth=0.25ex]
{ X & & X'\\
\\
Y && Y\\
&Z\\
};
\path[->,font=\scriptsize]
(m-1-1) edge node[auto] {$\varphi$} (m-1-3)
        edge node[auto,swap] {$f$} (m-3-1)
(m-1-3) edge node[auto] {$f'$} (m-3-3)
(m-3-1) edge[dotted] node[auto] {$\psi$} (m-3-3)
        edge node[auto,swap] {$g$} (m-4-2)
(m-3-3) edge node[auto] {$g$} (m-4-2);

\end{tikzpicture}
\end{center}

\begin{proof}
Choose $z\in Z,\,y\in g^{-1}(z),\,x\in f^{-1}(y)$. If we denote $x'\coloneqq \varphi(x), y'\coloneqq f'(x')$, then by hypothesis $y'\in g^{-1}(z)$. So by normality of $g$, there is a deck transformation $\psi\in\Deck(g)$ such that $\psi(y)=y'$ (which is even unique). Of course, $\psi\circ f(x)=f'\circ\varphi(x)$, and we claim now that we have $\psi\circ f=f'\circ\varphi$ globally.

Consider the set $A\coloneqq \{a\in X\mid\psi\circ f(a)=f'\circ\varphi(a)\}$. Clearly $A\neq\emptyset$ because $x\in A$. Also, it is closed in $X$ because all the spaces are Hausdorff. We want to show now that $A$ is also open. Because $X$ is connected, this implies $A=X$ and finishes the proof.

So let $a\in A$, and let $g(f(a))\in U\subseteq Z$ be an admissible neighbourhood for both $g\circ f$ and $g\circ f'$ (and so particularly for $g$). Furthermore let $V\subseteq g^{-1}(U)$ be the connected component containing $f(a)$, and $W\subseteq f^{-1}(V)$ the one containing $a$. Denote $V'\coloneqq\psi(V)$, and by $W'$ denote the connected component of $f'^{-1}(V')$ containing $\varphi(a)$. As it is  not clear by hypothesis that $W'=W''\coloneqq\varphi(W)$, set $\tilde{W}'\coloneqq W'\cap W''$, which is still an open neighbourhood of $\varphi(a)$, and adjust the other neighbourhoods in the following way:
\[\tilde{W}\coloneqq\varphi^{-1}(\tilde{W}')\subseteq W,\:\:\tilde{V}\coloneqq f(\tilde{W}),\:\:\tilde{V}'\coloneqq f'(\tilde{W}').\]
Note that we still have $a\in\tilde{W}$, that all these sets are still open, that $\tilde{U}\coloneqq g(\tilde{V})=g(\tilde{V}')$, and that the latter is still an admissible neighbourhood for $g\circ f$ and $g\circ f'$. By construction, by restricting all the maps to these neighbourhoods we get a commutative pentagon of homeomorphisms, so in particular $(\psi\circ f)_{|\tilde{W}}=(f'\circ\varphi)_{|\tilde{W}}$, which finishes the proof.
\end{proof}

\section{Dessins d'enfants and Belyi's theorem}\label{sec:dessins}

Let us now establish the basic theory of \emph{dessins d'enfants}. First, we give their definition and explain Grothendieck's equivalence to \emph{Belyi pairs}. Then, we state Belyi's famous theorem to establish the action of $\absGal$ on dessins and introduce the notions of fields of moduli and fields of definitions of the appearing objects.

We will mainly use the language of schemes in the context of algebraic curves, which seems to be the natural viewpoint here in the eyes of the author. Our notation will stay within bounds of the ones in the beautiful works \cite{wo}, \cite{schneps} and \cite{koe}, in which the curious reader will find many of the details omitted here.
\subsection{Dessins and Belyi morphisms}
\begin{defi}\label{def:dessin}
\begin{aenum}
\item A \emph{dessin d'enfant} (or \textit{Grothendieck dessin}, or \textit{children's drawing}) of degree $d$ is a tuple $(B,W,G,S)$, consisting of:
\begin{itemize}
\item A compact oriented connected real $2$-dimensional (topological) manifold $S$,
\item two finite disjoint subsets $B,W\subset S$ (called the \emph{black} and \emph{white vertices}),
\item an embedded graph $G\subset S$ with vertex set $V(G)=B\cupdot W$ which is bipartite with respect to that partition of $V(G)$, such that $S\setminus G$ is homeomorphic to a finite disjoint union of open discs (called the \emph{cells} of the dessin, and such that $|\pi_0(G\setminus(B\cup W))|=d$.
\end{itemize}
\item An isomorphism between two dessins $D\coloneqq (B,W,G,S)$ and $D'\coloneqq (B',W',G',S')$ is an orientation preserving homeomorphism $f\colon S\to S'$, such that
\[f(B)=B',\,f(W)=W',\text{ and }f(G)=G'.\]
\item By $\Aut(D)$ we denote the group of automorphisms of $D$, i.e.\ the group of isomorphisms between $D$ and itself.
\end{aenum}
\end{defi}

So, from a na\"ive point of view, a dessin is given by drawing several black and white dots on a surface and connecting them in such a manner by edges that the cells which are bounded by these edges are simply connected. The starting point of the theory of dessins d'enfants is that there are astonishingly many ways of giving the data of a dessin up to isomorphism. In the following proposition, we will survey several of these equivalences.

\begin{prop}\label{prop:dessin-equiv}
Giving a dessin in the above sense up to isomorphism is equivalent to giving each of the following data:
\begin{aenum}
\item A finite topological covering $\beta\colon  X^*\to \PeC\setminus\{0,\,1,\,\infty\}$ of degree $d$ up to equivalence of coverings.
\item A conjugacy class of a subgroup $G\leq \pi_1(\PeC\setminus\{0,\,1,\,\infty\})$ of index $d$.
\item A pair of permutations $(p_x,\,p_y)\in S_d^2$, such that $\langle p_x,\,p_y\rangle\leq S_d$ is a transitive subgroup, up to simultaneous conjugation in $S_d$.
\item A non-constant holomorphic map $\beta\colon  X\to \PeC$ of degree $d$, where $X$ is a compact Riemann surface and $\beta$ is ramified at most over the set $\{0,\,1,\,\infty\}$, up to fibre preserving biholomorphic maps.
\item A non-constant morphism $\beta\colon  X\to \PeC$ of degree $d$, where $X$ is a non-singular connected projective curve over $\CC$ and $\beta$ is ramified at most over the set $\{0,\,1,\,\infty\}$, up to fibre preserving $\CC$-scheme isomorphisms. Such a morphism is called a \emph{Belyi morphism} or \emph{Belyi pair}.
\end{aenum}
\end{prop}

\begin{proof}[Sketch of a proof]
The crucial point is the equivalence between an isomorphism class of dessins in the sense of the definition, and a conjugacy class of a pair of permutations as in c). It is shown by C.\ Voisin and J.\ Malgoire in \cite{mv}, and, in a slight variation, by G.\ Jones and D.\ Singerman in \cite[§ 3]{js}.

The equivalence of a), b) and c) is a simple consequence of Proposition \ref{prop:coverings}, as the fundamental group of $\PeC\setminus\{0,\,1,\,\infty\}$ is free in two generators $x$ and $y$ and the two permutations $p_x$ and $p_y$ describe their images under the monodromy map.

The equivalence between a) and d) is well-known in the theory of Riemann surfaces: The complex structure on $\PeC$ is unique, and every topological covering of $\PeC$ minus a finite set gives rise to a unique holomorphic covering by pulling back the complex structure, which can be uniquely compactified by Riemann's theorem on removing singularities.

Finally, the equivalence between d) and e) follows from the well-known GAGA principle first stated in \cite{gaga}.
\end{proof}

Let us illustrate these equivalences a little more: First, note that reconstructing a dessin in the sense of  Definition \ref{def:dessin} from d) can be understood in the following explicit way: As $S$, we take of course the Riemann surface $X$, as $B$ and $W$ we take the preimages of $0$ and $1$, respectively, and for the edges of $G$ we take the preimages of the open interval $(0,1)\subset\PeC$. Then, $S\setminus G$ is the preimage of the set $\PeC\setminus[0,\,1]$, which is open and simply connected. So the connected components of $S\setminus G$ are open and simply connected proper subsets of a compact surface and thus homeomorphic to an open disc.

Second, this indicates how to get from a dessin to the monodromy of the corresponding covering: As generators of $\pi_1(\PeC\setminus\{0,\,1,\,\infty\})$, we fix simple closed curves around $0$ and $1$ with winding number $1$ (say, starting in $\frac{1}{2}$), and call them $x$ and $y$, respecively. Choose a numbering of the edges of the dessin. Then, $p_x$ consists of the cycles given by listing the edges going out of each black vertex in counter-clockwise direction, and we get $p_y$ in the same manner from the white vertices.

Before we continue, we will state an easy consequence of Proposition \ref{prop:dessin-equiv}:
\begin{kor}\label{cor:dessins-finite-number}
For any $d\in \NN$, there are only finitely many dessins d'enfants of degree $d$ up to equivalence.
\end{kor}
\begin{proof} By Proposition \ref{prop:dessin-equiv}, a dessin can be characterised by a pair of permutations $(p_x,\,p_y)\in (S_d)^2$. So, $(d!)^2$ is an upper bound for the number of isomorphism classes of dessins of degree $d$.
\end{proof}

Next, we will establish the notion of a weak isomorphism between dessins. 
\begin{defi} We call two Belyi morphisms $\beta\colon X\to\PeC$ and $\beta'\colon X'\to\PeC$ \textit{weakly isomorphic} if there are biholomorphic maos $\varphi\colon X\to X'$ and $\psi\colon \PeC\to\PeC$ such that the following square commutes:
\begin{center}
\begin{tikzpicture}[description/.style={fill=white,inner sep=2pt}]
\matrix (m) [matrix of math nodes, row sep=2.5em,
column sep=2.25em, text height=1.5ex, text depth=0.25ex]
{ X & X'\\
 \PeC & \PeC\\};
\path[->,font=\scriptsize]
(m-1-1) edge node[auto] {$\varphi $} (m-1-2)
        edge node[left] {$\beta $} (m-2-1)
(m-1-2) edge node[auto] {$\beta' $} (m-2-2)
(m-2-1) edge node[auto] {$\psi$} (m-2-2);

\end{tikzpicture}
\end{center}
\end{defi}
Note that in the above definition, if $\beta$ and $\beta'$ are ramified exactly over $\{0,\,1,\,\infty\}$, then $\psi$ has to be a Möbius transformation fixing this set. This subgroup $W\leq\text{Aut}(\PeC)$ is clearly isomorphic to $S_3$ and generated by
\[s\colon z\mapsto 1-z\;\text{ and }\;t\colon z\mapsto z^{-1}.\]
In the case of two branch points, $\psi$ can of course still be taken from that group.  So for a dessin $\beta$, we get up to isomorphism all weakly isomorphic dessins by postcomposing with all elements of $W$. Let us reformulate this on a more abstract level:
\begin{defbem}\label{defbem:weak-action}
\begin{aenum}
\item The group $W$ acts on the set
of dessins from the left by $w\cdot\beta\coloneqq w\circ\beta$. The orbits under that action are precisely the weak isomorphism classes of dessins.
\item For a dessin $\beta$ we call (by slight abuse of the above definition, as a weak isomorphism should consist of two morphisms) $W_\beta\coloneqq \Stab(\beta)$, its stabiliser in $W$, the \textit{group of weak automorphisms}.
\item If a dessin $\beta$ is given by a pair of permutation $(p_x,\,p_y)$, then its images under the action of $W$ are described by the following table (where $p_z\coloneqq p_x^{-1}p_y^{-1})$:
\begin{center}
	\begin{tabular}{|c|c|c|c|c|c|}
	    \hline
		$\beta$ &   $s\cdot\beta$ &   $t\cdot\beta$ &   $(s\circ t)\cdot\beta$ &  $(t\circ s)\cdot\beta$ &  $(t\circ s \circ t)\cdot\beta$ \\\hline\hline
		$(p_x,\,p_y)$ & $(p_y,\,p_x)$ & $(p_z,\,p_y)$ & $(p_y,\,p_z)$ & $(p_z,\,p_x)$ & $(p_x,\,p_z)$ \\\hline
	\end{tabular}
\end{center}
\end{aenum}
\end{defbem}
\begin{proof}
Part a) was already discussed above. Proving c) amounts to checking what $s$ and $t$ do, and then composing and using \ref{satz:comp}, for example. The calculation has been done in \cite[2.5]{si}.
\end{proof}

Let us now explain the common notions of \emph{pre-clean} and \emph{clean} dessins and introduce a term for particularly un-clean ones:
\begin{defbem}\label{defbem:clean-dessins}
Let $\beta$ be a dessin defined by a pair of permutations $(p_x,\,p_y)$.
\begin{aenum}
\item $\beta$ is called \textit{pre-clean} if $p_y^2=1$, i.e.\ if all white vertices are either of valence $1$ or $2$.
\item $\beta$ is called \textit{clean} if all preimages of $1$ are ramification points of order precisely $2$, i.e.\ if all white vertices are of valence $2$.
\item If $\beta\colon X\to\PeC$ is a Belyi morphism of degree $d$, then if we define $a(z)\coloneqq 4z(1-z)\in\QQ[z]$ we find that $a\circ\beta$ is a clean dessin of degree $2d$.
\item We will call $\beta$ \textit{filthy} if it is not weakly isomorphic to a pre-clean dessin, i.e.\ $1\notin\{p_x^2,\,p_y^2,\,p_z^2\}$.
\end{aenum}
\end{defbem}
Another common class of dessins consists of the unicellular ones. We briefly discuss them here.
\begin{defbem}\label{defbem:uni-dessins}
\begin{aenum}
\item A dessin d'enfant $D$ is said to be \textit{unicellular} if it consists of exactly one open cell.
\item If $D$ is represented by a pair of permutations $(p_x,\,p_y)$, it is unicellular iff $p_z=p_x^{-1}p_y^{-1}$ consists of exactly one cycle.
\item If $D$ is represented by a Belyi morphism $\beta\colon X\to\PeC$, it is unicellular iff $\beta$ has exactly one pole.
\item If $D$ is a dessin in genus $0$, it is unicellular iff its graph is a tree.
\end{aenum}
\end{defbem}

\subsection{Fields of definition, moduli fields and Belyi's theorem}\label{ss:modulifields}

Let us recall the notions of fields of definition and moduli fields of schemes, varieties and morphisms. Following the presentation of this material in \cite{koe}, we will use the language of schemes. 

By $\Spec$, denote the usual spectrum functor from commutative rings to affine schemes. Fix a field $K$. Remember that a $K$-scheme is a pair $(S,p)$, where $S$ is a scheme and $p\colon S\to \Spec(K)$ a morphism, called the \emph{structure morphism}. $(S,p)$ is called $K$-variety if $S$ is reduced and $p$ is a separated morphism of finite type. A morphism between $K$-schemes is a scheme morphism forming a commutative triangle with the structure morphisms. Denote the such obtained categories by $\text{Sch}/K$ and $\text{Var}/K$, respectively.

Keep in mind that $p$ is part of the data of a $K$-scheme $(S,p)$, and that changing it gives a different $K$-scheme even though the abstract scheme $S$ stays the same. This allows us to define an action of $\Aut(K)$ on $\text{Sch}/K$:
\begin{defbem}\label{defbem:aut-K-action}
Let $(S,p)$ be a $K$-scheme, and $\sigma\in\Aut(K)$.
\begin{aenum}
\item Define $(S,p)^\sigma\coloneqq (S,\Spec(\sigma)\circ p)$.
\item Mapping $(S,p)\mapsto(S,p)^\sigma$ defines a right action of $\Aut(K)$ on $\text{Sch}/K$. This restricts to an action on $\text{Var}/K$.
\end{aenum}
\end{defbem}

We are now able to define the terms field of definition and moduli field.
\begin{defi}\label{defi:fields}
\begin{aenum}
\item A subfield $k\subseteq K$ is called a\textit{field of definition} of a $K$-scheme ($K$-variety) $(S,p)$ if there is a $k$-scheme ($k$-variety) $(S',p')$ such that there is a Cartesian diagram
\begin{center}
\begin{tikzpicture}[description/.style={fill=white,inner sep=2pt}]
\matrix (m) [matrix of math nodes, row sep=1.5em,
column sep=1.25em, text height=1.5ex, text depth=0.25ex]
{ S && S'\\
&\square\\
 \Spec(K) && \Spec(k)\\
};
\path[->,font=\scriptsize]
(m-1-1) edge node[auto] {$ $} (m-1-3)
        edge node[auto] {$p$} (m-3-1)
(m-1-3) edge node[auto] {$p'$} (m-3-3)
(m-3-1) edge node[auto] {$\Spec(\iota)$} (m-3-3);
\end{tikzpicture}
\end{center}
where $\iota\colon k\to K$ is the inclusion. Alternatively, $(S,p)$ is said to be \textit{defined over} $k$ then.
\item For a $K$-scheme $(S,p)$, define the following subgroup $U(S,p)\leq\Aut(K)$:
\[U(S,p)\coloneqq \{\sigma\in\Aut(K)\mid (S,p)^\sigma\cong(S,p)\}\]
The \textit{moduli field} of $(S,p)$ is then defined to be the fixed field under that group:
\[M(S,p)\coloneqq K^{U(S,p)}\]
\end{aenum}
\end{defi}

We will also need to understand the action of $\Aut(K)$ on morphisms. We will start the bad habit of omitting the structure morphisms here, which the reader should amend mentally.
\begin{defi} Let $\beta\colon  S\to T$ be a $K$-morphism (i.e.\ a morphism of $K$-schemes) and $\sigma\in \Aut(K)$.
\begin{aenum}
\item The scheme morphism $\beta$ is of course also a morphism between the $K$-schemes $S^\sigma$ and $T^\sigma$. We denote this $K$-morphism by $\beta^\sigma\colon S^\sigma\to T^\sigma$.
\item Let $\beta'\colon S'\to T'$ be another $K$-morphism. Then we write $\beta\cong\beta'$ if there are $K$-isomorphisms $\varphi\colon S\to S'$ and $\psi\colon T\to T'$ such that $\psi\circ\beta=\beta'\circ\varphi$. Specifically, we have $\beta\cong\beta^\sigma$ iff there are $K$-morphisms $\varphi, \psi$ such that the following diagram (where we write down at least some structure morphisms) commutes:
\begin{center}
\begin{tikzpicture}[description/.style={fill=white,inner sep=2pt}]
\matrix (m) [matrix of math nodes, row sep=1.5em,
column sep=1.25em, text height=1.5ex, text depth=0.25ex]
{ S^\sigma && S\\
\\
T^\sigma && T\\
\\
 \Spec(K) && \Spec(K)\\
};
\path[->,font=\scriptsize]
(m-1-1) edge node[auto] {$\varphi$} (m-1-3)
        edge node[auto] {$\beta^\sigma$} (m-3-1)
(m-1-3) edge node[auto] {$\beta$} (m-3-3)
(m-3-1) edge node[auto] {$\psi$} (m-3-3)
        edge node[auto] {$p$} (m-5-1)
(m-3-3) edge node[auto] {$p$} (m-5-3)
(m-5-1) edge node[auto] {$\Spec(\sigma)$} (m-5-3);
\end{tikzpicture}
\end{center}
\end{aenum}
\end{defi}

Let us now define the field of definition and the moduli field of a morphism in the same manner as above. For an inclusion $\iota\colon k\to K$ denote the corresponding base change functor by $\cdot\times_{\Spec(k)}\Spec(K)$.
\begin{defi}\label{defi:fields-morph}
\begin{aenum}
\item A morphism $\beta\colon S\to T$ of $K$-schemes (or $K$-varieties, respectively) is said to be \textit{defined over} a field $k\subseteq K$ if there is a morphism $\beta'\colon S'\to T'$ of $k$-schemes ($k$-varieties) such that $\beta\cong\beta'\times_{\Spec(k)}\Spec(K)$.
\item For a morphism $\beta\colon S\to T$ of $K$-schemes, define the following subgroup $U(\beta)\leq\Aut(K)$:
\[U(\beta)\coloneqq \{\sigma\in\Aut(K)\mid \beta^\sigma\cong\beta\}\]
The \textit{moduli field} of $\beta$ is then defined to be the fixed field under this group:
\[M(\beta)\coloneqq K^{U(\beta)}\]
\end{aenum}
\end{defi}
If, now, $\beta\colon X\to\PeC$ is a Belyi morphism, the above notation gives already a version for a moduli field of $\beta$. But this is not the one we usually want, so we formulate a different version here:
\begin{defi}\label{defi:dessinmoduli} Let $\beta\colon X\to \PeC$ be a Belyi morphism, and let $U_\beta\leq\Aut(\CC)$ be the subgroup of field automorphisms $\sigma$ such that there exists a $\CC$-isomorphism $f_\sigma\colon X^\sigma\to X$ such that the following diagram commutes:
\begin{center}
\begin{tikzpicture}[description/.style={fill=white,inner sep=2pt}]
\matrix (m) [matrix of math nodes, row sep=1.5em,
column sep=1.25em, text height=1.5ex, text depth=0.25ex]
{ X^\sigma && X\\
\\
 (\PeC)^\sigma && \PeC\\
};
\path[->,font=\scriptsize]
(m-1-1) edge node[auto] {$f_\sigma $} (m-1-3)
        edge node[auto] {$\beta^\sigma$} (m-3-1)
(m-1-3) edge node[auto] {$\beta$} (m-3-3)
(m-3-1) edge node[auto] {$\text{Proj}(\sigma)$} (m-3-3);
\end{tikzpicture}
\end{center}
where $\text{Proj}(\sigma)$ shall denote the scheme (not $\CC$-scheme!) automorphism of $\PeC=\text{Proj}(\CC[X_0,\,X_1])$ associated to the ring (not $\CC$-algebra!) automorphism of $\CC[X_0,\,X_1]$ which extends $\sigma\in\Aut(\CC)$ by acting trivially on $X_0$ and $X_1$.

Then, we call the fixed field $M_\beta\coloneqq \CC^{U_\beta}$ the \textit{moduli field} of the dessin corresponding to $\beta$.
\end{defi}
The difference to Definition \ref{defi:fields-morph} b) is that there, we allow composing $\text{Proj}(\sigma)$ with automorphisms of $\PeC$, potentially making the subgroup of $\Aut(\CC)$ larger and therefore the moduli field smaller. Let us make that precise, and add some more facts about all these fields, by citing \cite[Proposition 6]{wo}:
\begin{prop}\label{prop:fields-properties}
 Let $K$ be a field, and $\beta\colon S\to T$ a morphism of $K$-schemes, or of $K$-varieties. Then:
\begin{aenum}
\item $M(S)$ and $M(\beta)$ depend only on the $K$-isomorphism type of $S$ resp. $\beta$.
\item If furthermore $\beta$ is a Belyi morphism, then the same goes for $M_\beta$.
\item Every field of definition of $S$ (resp. $\beta$) contains $M(S)$ (resp. $M(\beta)$).
\item We have $M(S)\subseteq M(\beta)$.
\item If $\beta$ is a Belyi morphism, then we also have $M(\beta)\subseteq M_\beta$.
\item In this case $M_\beta$ (and therefore also $M(S)$ and $M(\beta)$) is a number field, i.e.\ a finite extension of $\QQ$.
\end{aenum}
\end{prop}
\begin{proof}
Parts a) to e) are direct consequences of the above definitions, and for f) we note that surely if $\sigma\in\Aut(\CC)$ then $\deg(\beta)=\deg(\beta^\sigma)$. So by Corollary \ref{cor:dessins-finite-number}, we have $[\Aut(\CC):U_\beta]=|\beta\cdot\Aut(\CC)|<\infty$ and so $[M_\beta:\QQ]<\infty$.
\end{proof}

We can now state Belyi's famous theorem:
\begin{satz}[V.\ G.\ Belyi] \label{satz:belyi}
Let $C$ be a smooth projective complex curve. $C$ is definable over a number field if and only if it admits a Belyi morphism $\beta\colon C\to\PeC$.
\end{satz}

In our scope, the gap between Proposition \ref{prop:fields-properties} f) and the “if” part (traditionally called the “obvious” part) of the theorem can be elegantly filled by the following theorem that can be found in \cite{hh}:

\begin{satz}[H.\ Hammer, F.\ Herrlich] Let $K$ be a field, and $X$ be a curve over $K$. Then $X$ can be defined over a finite extension of $M(X)$.
\end{satz}

The “only if” (traditionally called “trivial”) part of the theorem is a surprisingly explicit calculation of which several variations are known. The reader may refer to Section 3.1 of \cite{ggd}. 

\subsection{The action of $\absGal$ on dessins}

Given a Belyi morphism $\beta\colon X\to\PeC$ and an automorphism $\sigma\in\Aut(\CC)$, in fact $\beta^\sigma\colon X^\sigma\to\PeC$ is again a Belyi morphism, as the number of branch points is an intrinsic property of the underlying scheme morphism $\beta$ that is not changed by changing the structure morphisms. So, $\Aut(\CC)$ acts on the set of Belyi morphisms, and so, due to Proposition \ref{prop:dessin-equiv}, on the set of dessins d'enfants. By Belyi's theorem, this action factors through $\absGal$. It turns out that this action is faithful. Let us state this well known result in the following

\begin{satz}\label{satz:dessinaction}
For every $g\in\NN$, the action of $\absGal$ on the set of dessins d'enfants of genus $g$ is faithful. This still holds for every $g$ if we restrict to the clean, unicellular dessins in genus $g$.
\end{satz}

In genus $1$, this can be seen easily as the action defined above is compatible with the $j$-invariant, i.e.\ we have $\sigma((j(E)))=j(E^\sigma)$ for an elliptic curve $E$. It was noted by F. Armknecht in \cite[Satz 3.11]{arm} that this argument generalises to higher genera, by even restricting to hyperelliptic curves only.

The standard proof for genus $0$ can be found in \cite[Thm.\ II.4]{schneps}, where it is attributed to H. W. Lenstra, Jr. The result there even stronger than the claim here, as it establishes the faithfulness of the action even on trees.

The fact that the faithfulness does not break when restricting to unicellular and clean dessins can be proven easily by carefully going through the proof of the “trivial” part of Belyi's theorem—see \cite[Proof of Theorem G]{diss} for details.

\section{Origamis and their Teichmüller curves}\label{sec:origamis}

\subsection{Origamis as coverings}
Here, we will the exhibit the definition of an \emph{origami} first in the spirit of the previous section and then in the scope of translation surfaces. We will then present a short survey of the theory of translation surfaces and Teichmüller curves, and finally speak about arithmetic aspects of all these objects.

The term “origami” was coined by P.\ Lochak—see \cite{loc} and be aware that it is used in a slightly different meaning there.

The standard intuition for constructing an origami of degree $d\in\NN$ is the following: Take $d$ copies of the unit square $[0,\,1]\times[0,\,1]$ and glue upper edges to lower edges and left to right edges, respecting the orientation, until there are no free edges left, in a way that we do not end up with more than one connected component. In this way, we get a compact topological surface $X$ together with a tiling into $d$ squares (hence the other common name for origamis: \emph{square tiled surfaces}).

Such a tiling naturally defines a (ramified) covering $p\colon X\to E$ of the unique origami of degree $1$, which we call $E$, i.e.\ a compact surface of genus $1$, by sending each square of $X$ to the unique square of $E$. $p$ is ramified over one point, namely the image under glueing of the vertices of the square. 

To be more exact, fix $E\coloneqq \CC/\ZZ^2$ (this defines a complex structure on $E$), then $p$ is ramified (at most) over $0+\ZZ^2$. Denote this point by $\infty$ for the rest of this work, and furthermore $E^\ast\coloneqq E\setminus\{\infty\},\: X^\ast\coloneqq p^{-1}(E^\ast)$. Note that $p\colon X^\ast\to E^\ast$ is an unramified covering of degree $d$, and the fundamental group of $E^\ast$ is free in two generators, so one should expect analogies to the world of dessins. Let us write down a proper definition:

\begin{defi}\label{defi:origami}
\begin{aenum}
\item An \textit{origami} $O$ of degree $d$ is an unramified covering $O\coloneqq (p\colon X^*\to E^*)$ of degree $d$, where $X^*$ is a (non-compact)  topological surface.
\item If $O'=(p'\colon X'^*\to E^*)$ is another origami, then we say that $O$ is \textit{equivalent to} $O'$ (which we denote by $O\cong O'$), if the defining coverings are isomorphic, i.e.\ if there is a homeomorphism $\varphi\colon X'^*\to X^*$ such that $p'=p\circ\varphi$.
\item $O=(p\colon X^*\to E^*)$ is called \textit{normal} if $p$ is a normal covering.
\end{aenum}
\end{defi}
Like in the case of dessins, this is not the only possible way to define an origami. We list several others here:
\begin{prop}\label{prop:ori-equiv}
Giving an origami of degree $d$ in the above sense up to equivalence is equivalent to giving each of the following data:
\begin{aenum}
\item A conjugacy class of a subgroup $G\leq\pi_1(E^*)\cong F_2$ of index $d$.
\item A pair of permutations $(p_A,\,p_B)\in S_d^2$, such that $\langle p_A,\,p_B\rangle\leq S_d$ is a transitive subgroup, up to simultaneous conjugation in $S_d$.
\item A non-constant holomorphic map $p\colon  X\to E$ of degree $d$, where $X$ is a compact Riemann surface and $p$ is ramified at most over the set $\{\infty\}$, up to fibre preserving biholomorphic maps.
\item A non-constant morphism $\colon  X\to E$ of degree $d$, where $X$ is a non-singular connected projective curve over $\CC$ and $p$ is ramified at most over the set $\{\infty\}$, up to fibre preserving isomorphisms.
\end{aenum}
\end{prop}

The proof is completely analogous to the one of Proposition \ref{prop:dessin-equiv}. See \cite[Proposition 1.2]{kk} for details.

\subsection{Origamis and translation surfaces}
As we want to study origamis as translation surfaces, let us briefly recall their theory.

\begin{defi}\label{defi:trans-structure}
Let $X$ be a Riemann surface, and let $\mathfrak{X}$ be its complex structure.
\begin{aenum}
\item A \textit{translation structure} $\mu$ on $X$ is an atlas compatible with $\mathfrak{X}$ (as real analytic atlases, i.e.\ their union is an atlas of a real analytic surface), such that for any two charts $f,\,g\in\mu$, the transition map is locally a translation, i.e.\ a map
\[\varphi_{f,g}\colon U\subseteq\CC\to U'\subseteq\CC,\,x\mapsto x+t_{f,g}\]
for some $t_{f,g}\in\CC$. We call the pair $X_\mu\coloneqq (X,\,\mu)$ a translation surface.
\item A biholomorphic map $f\colon X_\mu\to Y_\nu$ between translation surfaces is called a \textit{translation}, or an \textit{isomorphism of translation surfaces}, if it is locally (i.e.\ on the level of charts) a translation. $X_\mu$ and $Y_\nu$ are then called isomorphic (as translation surfaces). If we have furthermore $X=Y$, we call the translation structures $\mu$ and $\nu$ are \textit{equivalent}.
\item If $\mu$ is a translation structure on $X$, and $A\in \SL_2(\RR)$, then we define the translation structure
\[A\cdot\mu\coloneqq \{A\cdot f\mid f\in \mu\}\]
where $A$ shall act on $\CC$ by identifying it with $\RR^2$ as usual. Therefore, we get a left action of $\SL_2(\RR)$ on the set of translation structures on $X$.
\end{aenum}
\end{defi}

Keep in mind that a translation structure $\mu$ on $X$, seen as a complex structure, is usually \emph{not equivalent} to $\mathfrak{X}$!

Let us go on by defining affine diffeomorphisms and the notion of the Veech group of a translation surface:

\begin{defbem}\label{defbem:veechgroup}
Let $X_\mu,\, Y_\nu$ be translation surfaces.
\begin{aenum}
\item An \textit{affine diffeomorphism} $f\colon X_\mu\to Y_\nu $ is an orientation preserving diffeomorphism such that locally (i.e.\ when going down into the charts) it is a map of the form
\[x\mapsto A\cdot x+ t,\, A\in \GL_2(\RR),\,t\in \CC.\]
We call $X_\mu$ and $Y_\nu$ \textit{affinely equivalent} if there is such an affine diffeomorphism.
\item The matrix $A\eqqcolon A_f$ in a) actually is a global datum of $f$, i.e.\ it is the same for every chart. We write $\text{der}(f)\coloneqq A_f$.
\item An affine diffeomorphism $f$ is a translation iff $A_f=I$.
\item If $g\colon Y_\nu\to Z_\xi$ is another affine diffeomorphism, then $\text{der}(g\circ f)=\text{der}(g)\cdot\text{der}(f)$. In particular, $\text{der}\colon \text{Aff}^+(X_\mu)\to \GL_2(\RR)$, is a group homomorphism.
\item We denote the group of all affine orientation preserving diffeomorphisms from $X_\mu$ to itself by $\text{Aff}^+(X_\mu)$.
\item $\text{Trans}(X_\mu)\coloneqq \ker(\text{der})$ is called the \textit{group of translations} of $X_\mu$.
\item $\Gamma(X_\mu)\coloneqq \im(\text{der})$ is called the \textit{Veech group} of $X_\mu$. Its image under the projection map $\GL_2(\RR)\to\text{PGL}_2(\RR)$ is called the \textit{projective Veech group} of $X_\mu$. We denote it by $\text{P}\Gamma(X_\mu)$.
\item If $A\in\SL_2(\RR)$, then we have $A\in\Gamma(X_\mu)\Leftrightarrow X_\mu\cong X_{A\cdot\mu}$ as translation surfaces.
\end{aenum}
\end{defbem}

For a discussion of this, see \cite[Section 1.3]{gs}.

Note that if $X_\mu$ has finite volume, every affine diffeomorphism has to preserve the volume, and as we require affine affine diffeomorphisms also to preserve the orientation, this yields $\Gamma(X_\mu)\subseteq \SL_2(\RR)$.

Our model genus $1$ surface $E=\CC/\ZZ^2$ carries a natural translation structure $\mu_0$, as $\ZZ^2$ acts on $\CC$ by translations. It is quite easy to see that the Veech group $\Gamma(E_{\mu_0})$ is the modular group $\SL_2(\ZZ)$: Denote by $\pi\colon \CC\to E$ the projection, then every affine diffeomorphism on $E$ lifts to a globally affine transformation on $\CC$ via $\pi$. On the other hand, a matrix $A\in\SL_2(\RR)$ induces an affine diffeomorphism on $E$ if and only if it respects the lattice $\ZZ^2$, i.e.\ iff $A\in\SL_2(\ZZ)$. Analogously, for $B\in\SL_2(\RR)$, the veech group of $\CC/(B\cdot \ZZ^2)$ (with a translation structure $\mu_B$ obtainend in the same way as above) is $B\SL_2(\ZZ)B^{-1}$. Also, we have $\Gamma(E_{B\cdot \mu_0})=B\SL_2(\ZZ)B^{-1}$.

Let us fix our favourite generators for $\SL_2(\ZZ)$:
\[S\coloneqq \begin{pmatrix}0 & -1\\ 1& 0\end{pmatrix}\text{ and }T\coloneqq\begin{pmatrix}1 & 1\\ 0& 1\end{pmatrix}.\]

Given a translation surface $X_\mu$ and a topological covering $f\colon Y\to X$, we obtain a translation structure $f^\ast\mu$ on $Y$ by precomposing small enough charts from $\mu$ by $f$. The map $p\colon Y_{f^*\mu}\to X_\mu$ is then usually called \emph{translation covering}. In this way, let us define the Veech group of an origami:

\begin{defi}\label{defi:ori-veechgroup}
Let $O=(p\colon X^*\to E^*)$ be an origami. Then we call 
\[\Gamma(O)\coloneqq \Gamma(X^*_{p^*\mu_0})\]
the \textit{Veech group} of $O$.
\end{defi}

Let us list some fundamental properties:

\begin{prop}\label{prop:ori-veechgroup}
 If $O=(p\colon X^*\to E^*)$ is an origami, and $\Gamma\coloneqq\Gamma(O)$ its Veech group, then we have:
\begin{aenum}
\item $\Gamma\subseteq \SL_2(\ZZ)$.
\item $\Gamma(X^*,B\cdot p^*\mu_0)=B\Gamma B^{-1}$ for every $B\in \SL_2(\RR)$.
\item The isomorphism classes of origamis that are affinely equivalent to $O$ are in bijection with the left cosets of $\Gamma$ in $\SL_2(\ZZ)$.
\item $[\SL_2(\ZZ):\Gamma]<\infty$.
\end{aenum}
\end{prop}

Proofs can be found in \cite{gs}: a) and b) can be found in Section 1.3 there; c) is elementary if we use Schmithüsen's Proposition 3.3 which states that any affine diffeomorphism $X^\ast\to X'^\ast$ descends to an affine diffeomorphism $E^\ast\to E^\ast$. Her Corollary 3.6 provides a proof for d). Though, part d) can be proven in a more elementary way, by noting that the $\SL_2(\ZZ)$ action on origamis $O=(p\colon X^\ast \to E^\ast)$ preserves the volume of $X^\ast$ and thus the degree of $p$. We can then conclude with the same argument as in (our) Corollary \ref{cor:dessins-finite-number}.

To calculate the Veech groups of the special origamis appearing later in this work, we use a rather different characterisation of Veech groups of origamis found by G.\ Weitze-Schmithüsen in \cite{gs}. Remember $\pi_1(E^\ast)\cong F_2$ and consider the group homomorphism $\varphi\colon\Aut(\pi_1(E^\ast))\to\Out(\pi_1(E^\ast))\cong\GL_2(\ZZ)$. Via the latter isomorphism, we define the “orientation preserving” (outer) automorphism groups $\Out^+(\pi_1(E^\ast))\coloneqq \SL_2(\ZZ)$ and $\Aut^+(\pi_1(E^\ast))\coloneqq\varphi^{-1}(\Out^+(\pi_1(E^\ast)))$.

\begin{satz}[G. Schmithüsen]\label{satz:gabi-vg}
Let $O\coloneqq(p\colon X^\ast\to E^\ast)$ be an origami. Then we have: 
\begin{aenum}
\item $\Gamma(O)=\varphi(\Stab(p_\ast\pi_1(X^\ast)))$.
\item If $f\in\Aut^+(\pi_1(E^\ast))$ with $\varphi(f)=A\in\SL_2(\ZZ)$, and the monodromy of $O$ is given by $m_p$, then the monodromy of $A\cdot O$ is given by $m_p\circ f$.
\end{aenum}
\end{satz}
The first part is Theorem 1 in said work, the second is the isomorphism $\beta$ from Proposition 3.5 there.

\subsection{Moduli and Teichmüller spaces of curves}
We begin by giving a somewhat rough definition of different versions of the (coarse) moduli space of compact Riemann surfaces. A very detailed reference on this subject is provided in \cite{ha}.

\begin{defi}\label{defi:moduli-space}
\begin{aenum}
\item Define the \textit{coarse moduli space of Riemann surfaces of genus $g$ with $n$ distinguished marked points} as
\[M_{g,n}\coloneqq \left\{(X,\,p_1,\ldots,p_n)\mid X \text{ \scriptsize compact R. s. of genus }g,\, p_i\in X,\, p_i\neq  p_j \text{ \scriptsize for } i\neq j\right\}/_\sim\]
where $(X,\,p_1,\ldots,p_n)\sim(Y,\,q_1,\ldots,q_n)$ if there is a biholomorphic map $\varphi\colon X\to Y$ with $\varphi(p_i)=q_i,\,i=1,\ldots,n$.
\item Define the \textit{coarse moduli space of Riemann surfaces of genus $g$ with $n$ non-distinguished marked points} as
\[M_{g,[n]}\coloneqq \left\{(X,\,p_1,\ldots,p_n)\mid X \text{ \scriptsize compact R. s. of genus }g,\, p_i\in X,\, p_i\neq  p_j \text{ \scriptsize for } i\neq j\right\}/_\sim\]
where $(X,\,p_1,\ldots,p_n)\sim(Y,\,q_1,\ldots,q_n)$ if there is a biholomorphic function $\varphi\colon X\to Y$ and a permutation $\pi\in S_n$, such that $\varphi(p_i)=q_{\pi(i)},\,i=1,\ldots,n.$
\item Finally, define the \textit{coarse moduli space of Riemann surfaces of genus $g$} as
\[M_g\coloneqq M_{g,0}=M_{g,[0]}.\]
\end{aenum}
\end{defi}
In fact, $M_{g,n}$ and $M_{g,[n]}$, which we defined just as sets, can be turned into complex quasi-projective varieties, or complex analytic spaces, of dimension $3g-3+n$ (whenever this expression is positive---we have $\dim(M_{1,0})=1$, and $\dim(M_{0,n})=0$ for $n\leq 3$). There are natural projections
\[M_{g,n}\to M_{g,[n]}\to M_g\]
by forgetting the order of the marked points, and totally forgetting the marked points.

All these versions also exist as schemes, which are all defined over $\ZZ$.

The usual analytical approach to understanding moduli spaces is Teichmüller theory. Let us recall the basic facts. We begin by giving the definition of Teichmüller spaces:
\begin{defbem}\label{defbem:teichmüller}
Let $S$ be a fixed compact Riemann surface of genus $g$ with $n$ marked points. (Let us write shortly that $S$ is of type $(g,n)$.)
\begin{aenum}
\item If $X$ is another surface of this type, a \textit{marking} on $X$ is an orientation preserving diffeomorphism $\varphi\colon S\to X$ which respects the marked points.
\item We define the \textit{Teichmüller space} of the surface $S$ as
\[\mathcal{T}(S)\coloneqq \left\{(X,\,\varphi)\mid X\text{ R.\,s. of type }(g,n),\,\varphi\colon S\to X \text{ a marking}\right\}/_\sim\]
where $(X,\,\varphi)\sim(Y,\,\psi)$ if $\psi\circ\varphi^{-1}\colon  X\to Y$ is homotopic to a biholomorphism respecting the marked points (where, of course, the homotopy shall fix the marked points).
\item If $S'$ is another surface of type $(g,n)$, then any choice of a marking $\varphi\colon S\to S'$ yields a bijection $\mathcal{T}(S')\to\mathcal{T}(S)$ by precomposing all markings with $\varphi$, which gives us the right to just write $\mathcal{T}_{g,n}$.
\end{aenum}
\end{defbem}
In the same manner as above, there exist also versions with non-ordered and without marked points, denoted by $\mathcal{T}_{g,[n]}$ and $\mathcal{T}_g$, respectively. It turns out that $\mathcal{T}_{g,n}$ is a complex manifold of dimension $3g-3+n$ whenever this expression is positive. Actually, it is isomorphic to a unit ball of that dimension. The group of orientation preserving diffeomorphisms of $S$, denoted by $\text{Diffeo}^+(S)$, acts on $\mathcal{T}(S)$ from the left by composition with the marking. It is clear that this action factors through the \emph{mapping class group} $\Sigma(S)\coloneqq \pi_0(\text{Diffeo}^+(S))$ and that its orbits are precisely the isomorphism types of Riemann surfaces of type $(g,n)$, so that we have
\[\mathcal{T}(S)/\Sigma(S)\cong M_{g,n}.\]
It is also true but far less obvious that $\Sigma(S)$ acts properly discontinous and with finite stabilisers, and that the above equation holds in the category of complex spaces.

Analogous statements hold for surfaces with $n$ non-ordered marked points. Note that compact surfaces of genus $g$ with finitely many points removed can be compactified uniquely and is thus naturally an element of $M_{g,[n]}$.

\subsection{Teichmüller discs and Teichmüller curves}
Let $X$ be a compact Riemann surface of genus $g$ with $n$ punctures, endowed with a translation structure. For $B\in\SL_2(\RR)$, denote by $X_B$ the Riemann surface that we get by endowing $X$ with the complex structure induced by $B\cdot\mu$. Then the identity map $\id\colon X=X_I\to X_B$ is a marking in the sense of Definition and Remark \ref{defbem:teichmüller} a). Note that this map is in general \textit{not} holomorphic! So we get a map
\[\theta\colon \SL_2(\RR)\to\mathcal{T}_{g,[n]},\,B\mapsto \left[(X_B,\,\id\colon X_I\to X_B)\right].\]
Since for $B\in \SL_2(\RR)$ we have that $z\mapsto B\cdot z$ is biholomorphic iff $B\in\text{SO}(2)$ it is easy to see that $\theta$ factors through $\text{SO}(2)\backslash\SL_2(\RR)\cong \Hp$. 
 We fix the latter bijection as $m\colon \text{SO}(2)\backslash\SL_2(\RR)\to\Hp,\,[A]\mapsto \overline{-A^{-1}(i)}$. The reason for this choice will become clear in a bit. The factor map
\[\overline{\theta}\colon \Hp \to \mathcal{T}_{g,[n]}\]
is injective. It is in fact biholomorphic to its image, and furthermore an isometry with respect to the standard hyperbolic metric on $\Hp$ and the Teichmüller metric on $\mathcal{T}_{g,[n]}$ as defined, for example, in \cite[6.4]{hub}. See \cite[2.6.5 and 2.6.6]{nag} for details. This leads to the following
\begin{defi}\label{defi:tm-disc}
Let $X_\mu\coloneqq (X,\,\mu)$ be a translation surface of type $(g,[n])$. Then, the isometric image
\[\Delta_{X_\mu}\coloneqq \overline{\theta}(\Hp)\subseteq\mathcal{T}_{g,[n]}\]
is called the \textit{Teichmüller disc} associated with $X_\mu$.
\end{defi}
The image of a Teichmüller disc $\Delta_{X_\mu}$ under the projection map into moduli space is, in general, not an algebraic subvariety. If the Veech group of the translation surface $X_{\mu}$ is a lattice in $\SL_2(\RR)$, i.e.\ if $\text{vol}(\Hp/\Gamma(X_\mu))<\infty$, then in fact the image of $\Delta_{X_\mu}$ in the moduli space is an algebraic curve, as stated in the following theorem. It is usually attributed to John Smillie, but cited from \cite{mcm}.

\begin{satz}[J.\ Smillie]\label{satz:tm-curve}
Let $X_\mu$ be a translation surface of type $(g,[n])$, and $\Delta_{X_\mu}$ its Teichmüller disc. Furthermore let $p\colon \mathcal{T}_{g,[n]}\to M_{g,[n]}$ be the projection. Then we have:
\begin{aenum}
\item $p(\Delta_{X_\mu})\subseteq M_{g,[n]}$ is an algebraic curve iff $\Gamma(X_\mu)$ is a lattice. It is then called the \textit{Teichmüller curve} associated to $X_\mu$.
\item In this case, the following diagram is commutative if we define $R\coloneqq \left(\begin{smallmatrix}-1& 0\\0&1\end{smallmatrix}\right)$:

\begin{center}
\begin{tikzpicture}[description/.style={fill=white,inner sep=2pt}]
\matrix (m) [matrix of math nodes, row sep=3.5em, column sep=3em, text height=1.5ex, text depth=0.25ex]
{ \Hp & \Delta_{X_{\mu}} & \mathcal{T}_{g,[n]}\\
\Hp/R\Gamma(X_{\mu})R^{-1}& p(\Delta_{X_{\mu}}) & M_{g,[n]}\\
};
\path[right hook->,font=\scriptsize] (m-1-1) edge node[auto] {$ \overline{\theta} $} (m-1-2);
\path[->,font=\scriptsize] (m-1-1) edge node[auto, swap] {$/R\Gamma(X_\mu)R^{-1}$} (m-2-1);
\path[right hook->,font=\scriptsize] (m-1-2) edge node[auto] {$ $} (m-1-3);
\path[->,font=\scriptsize] (m-1-2) edge node[auto] {$p_{|\Delta_{X_\mu}}$} (m-2-2);
\path[->,font=\scriptsize] (m-1-3) edge node[auto] {$p$} (m-2-3);
\path[->,font=\scriptsize] (m-2-1) edge node[auto] {$j$} (m-2-2);
\path[right hook->,font=\scriptsize] (m-2-2) edge node[auto] {$\iota$} (m-2-3);
\end{tikzpicture}
\end{center}
Furthermore, the map $j$ is the normalisation map for the algebraic curve $p(\Delta_{X_\mu})$.
\end{aenum}
\end{satz}
 A  complete proof of this theorem as it is stated here can be found in \cite[2.4]{hs} or \cite[Proposition 3.2]{loc}.

We have learned in \ref{prop:ori-veechgroup} that Veech groups of origamis have finite index in $\SL_2(\ZZ)$, so in particular they are lattices. We get: 
\begin{kor}\label{kor:origamicurve} Let $O = (f\colon X^*\to E^*)$ be an origami, where $X^*$ is of type $(g,[n])$. Then,
\[C(O)\coloneqq p(\Delta_{X_{f^*\mu_0}})\subseteq M_{g,[n]}\]
is an algebraic curve which we call the \textit{origami curve} defined by $O$.
\end{kor}

An origami curve does not determine a unique origami in general. In fact, we have: 
\begin{prop}\label{prop:aff-eq-origamis}
Let $O,\,O'$ be origamis. Then we have $C(O)=C(O')$ iff $O$ and $O'$ are affinely equivalent.
\end{prop}
A proof can be found in \cite[Proposition 5 b)]{hs2}.

\subsection{Strebel directions and cylinder decompositions}
We will not give any proofs here, for details see for example Sections 3.2 and 3.3 in \cite{kk} and Section 4 in \cite{hs}.

Note that translations on $\RR^2$ are isometries with respect to the standard Euclidean metric, so that on a translation surface $X$, we get a global metric by glueing together the local Euclidean metrics from the charts. It is called the \emph{flat metric} on $X$. A geodesic path $\gamma$ w.r.t.\ this metric is locally a straight line, i.e.\ it is locally of the form $t\mapsto t\cdot v+w$, where $v\in\RR^2$ is independent of the choice of charts. We call $v$ (or rather its equivalence class in $\Pe(\RR)$) the \emph{direction} of $\gamma$. We call $\gamma$ \emph{maximal} if its image is not properly contained in the image of another geodesic path.

\begin{defi}
Let $X_\mu$ be a translation surface of type $(g,\,n)$.
\begin{aenum}
\item A direction $v\in\Pe(\RR)$ is called \textit{Strebel} if every maximal geodesic path on $X_\mu$ with direction $v$ is either closed, or a \textit{saddle connection} (i.e.\ it connects two punctures of $X_\mu$).
\item We call two Strebel directions $v,\,v'\in\Pe(\RR)$ \textit{equivalent} if there is an $A\in \text{P}\Gamma(X_\mu)$ such that $A\cdot v=v'$.
\end{aenum}
\end{defi}

Now, if $v$ is a Strebel direction for $X_\mu$, a \textit{cylinder} in $X_\mu$ is the image of a homeomorphism $c\colon (0,\,1)\times S^1\to U\subseteq X_\mu$, where $U$ is an open subset of $X_\mu$,  with the condition that for every $s\in(0,\,1)$, the restriction to $\{s\}\times S^1$ is a closed geodesic. A cylinder is called \textit{maximal} if it is not properly contained in another cylinder. We note:
\begin{bem}
With the exception of the case $(g,\,n)=(1,\,0)$, the maximal cylinders of $X_\mu$ in the Strebel direction $v$ are the connected components of $X_\mu\setminus S$, where $S$ is the union of the images of all saddle connections in direction $v$.
\end{bem}

Let us restrict to origamis now and summarise the situation in this case:

\begin{prop}\label{prop:ori-strebel-cylinders}
Let $O=(p\colon X^*\to E^*)$ be an origami. Then we have:
\begin{aenum}
\item There is a bijection between the following sets:
    \begin{itemize}
    \item Equivalence classes of Strebel directions of $O$,
    \item Conjugacy classes of maximal parabolic subgroups in $\text{P}\Gamma(O)$,
    \item Punctures (called \textit{cusps}) of the normalisation of the origami curve, $\Hp/R\Gamma(O)R^{-1}$.
    \end{itemize}
\item The vector $\left(\begin{smallmatrix}1\\0\end{smallmatrix}\right)$ is a Strebel direction of $O$, called its \textit{horizontal} Strebel direction.
\item Any maximal parabolic subgroup of $\text{P}\Gamma(O)$ is generated by the equivalence class of a matrix of the form $gT^wg^{-1}$, for some $w\in\NN,\,g\in\SL_2(\ZZ)$.
\item The Strebel direction corresponding to a maximal parabolic subgroup $\langle \overline{gT^wg^{-1}}\rangle$ is  $v_g\coloneqq g\cdot\left(\begin{smallmatrix}1\\0\end{smallmatrix}\right)$. The maximal cylinders of $O$ in this Strebel direction are the maximal horizontal cylinders of the origami $g^{-1}\cdot O$.
\end{aenum}
\end{prop}
Proofs can be found in Sections 3.2 and 3.3 of \cite{kk}.

\subsection{The action of $\absGal$ on origami curves}
As the way we constructed Teichmüller curves is clearly of analytical nature, it may be surprising that they have interesting arithmetic properties. This kind of connection reminds of the Theorem of Belyi, which we can indeed use to prove a small part of the following

\begin{prop}\label{prop:ori-curves-arithmetic}
Let $O=(p\colon X^*\to E^*)$ be an origami and $C(O)\subseteq M_{g,[n]}$ its Teichmüller curve.
\begin{aenum}
\item Then, the normalisation map $j\colon \Hp/R^{-1}\Gamma(O)R\to C(O)$ and the inclusion $\iota\colon C(O)\hookrightarrow M_{g,[n]}$ are defined over number fields.
\item Let $\sigma\in\absGal$ be a Galois automorphism, and $O^\sigma=(p^\sigma\colon(X^*)^\sigma\to E^*)$ the Galois conjugate origami\footnote{Note here that we fixed the choice of $E=\CC/\ZZ^2$. We have $j(E)=1728\in\QQ$, and thus $E$ is defined over $\QQ$.}. Then we have\footnote{Strictly speaking, we should use a notation like $\iota_{C(O)}$ to distinguish the embeddings of different origami curves. We suppress the index for reasons of simplicity and bear in mind that the following formula has two different morphisms called $\iota$.}
\[\iota^\sigma((C(O))^\sigma)=\iota(C(O^\sigma))\subseteq M_{g,[n]},\]
so in particular $C(O^\sigma)\cong (C(O))^\sigma$.
\end{aenum}
\end{prop}
This result is proven by Möller in \cite[Proposition 3.2]{mm1}. The main ingredient to his proof is the fact that $\mathcal{H_E}$, the \textit{Hurwitz stack} of coverings of elliptic curves ramified over one prescribed point of some prescribed genus and degree, is a smooth stack defined over $\QQ$. This is a result due to Wewers that can be found in \cite{we}. One identifies (an orbifold version of) $C(O)$ as a geometrically connected component of $\mathcal{H_E}$. The $\absGal$-orbit of $C(O)$ consists precisely of the geometrically connected components of the $\QQ$-connected component containing $C(O)$. The definability of $C(O)$ over a finite extension of $\QQ$ then follows from the fact that $\mathcal{H_E}$ has only finitely many geometrically connected components. Showing this amounts to showing that the number or origamis of given degree is finite.

For a more detailed account of the proof of part b) than in \cite{mm1}, see \cite[Section 3.6]{diss}.

The desired Galois action on origami curves is a direct consequence of Proposition \ref{prop:ori-curves-arithmetic}:

\begin{kor}\label{cor:ori-curve-action}
Let $\mathfrak{O}$ be a set containing one origami of each isomorphism type. Then there is a natural right action of $\absGal$ on the set 
\[C(\mathfrak{O})\coloneqq \{C(O)\mid O\in\mathfrak{O}\},\]
where $\sigma\in \absGal$ sends $C(O)$ to $C(O^\sigma)$.
\end{kor}

\subsection{Galois invariants and moduli fields}

Let us now return to the notion of moduli fields, which we defined in \ref{ss:modulifields}. We begin by defining the moduli field of an origami, and of an origami curve:

\begin{defi}\label{def:ori-moduli}
Let $O=(p\colon X^*\to E^*)$ be an origami, and $C(O)$ its origami curve.
\begin{aenum}
\item Consider the following subgroup of $\Aut(\CC)$:
\[U(O)\coloneqq \left\{\sigma\in\Aut(\CC)\mid\exists\CC\text{-isomorphism }\varphi\colon X^\sigma\to X: p^\sigma=p\circ\varphi\right\}\]
Then, $M(O)\coloneqq \CC^{U(O)}$ is called the \textit{moduli field} of $O$.
\item Remember that $M_{g,[n]}$ is defined over $\QQ$ and define
\[U(C(O))\coloneqq \left\{\sigma\in\Aut(\CC)\mid\ C(O)=(C(O))^\sigma\right\},\]
where as usual we consider $C(O)$ and $(C(O))^\sigma$ as subsets of $M_{g,[n]}$. Then, we call $ M(C(O))\coloneqq \CC^{U(C(O))}$ the \textit{moduli field} of the origami curve $C(O)$.
\end{aenum}
\end{defi}

Let us first note some easy to prove properties of these moduli fields:

\begin{bem}\label{bem:ori-moduli} Let, again, $O=(p\colon X^*\to E^*)$ be an origami, and $C(O)$ its origami curve. Then we have:
\begin{aenum}
\item $M(C(O))\subseteq M(O)$.
\item $[M(O):\QQ]=|O\cdot\Aut(\CC)|<\infty$ and $[M(C(O)):\QQ]=|C(O)\cdot\Aut(\CC)|<\infty$.
\end{aenum}
\end{bem}

\begin{proof}
Part a) is a consequence of Corollary \ref{cor:ori-curve-action}: If $\sigma\in\Aut(\CC)$ fixes $O$, it particularly fixes $C(O)$.

For part b) we begin by noting that we have $|O\cdot\Aut(\CC)|<\infty$ because the degree of $O$ is an invariant under the action of $\Aut(\CC)$, and, as we have seen before, we can bound the number of origamis of degree $d$ by $(d!)^2$. Furthermore, we have $[\Aut(\CC):U(O)]=|O\cdot\Aut(\CC)|$, as $U(O)$ is the stabiliser of $O$ under the action of $\Aut(\CC)$. From \cite[Lemma 1.6]{koe} follows the equality $[M(O):\QQ]=[\Aut(\CC):U(O)]$, given that we can show that $U(O)$ is a closed subgroup of $\Aut(\CC)$. Remember that a subgroup $G\leq\Aut(\CC)$ is closed iff there is a subfield $F\subseteq \CC$ with $G=\Aut(\CC/F)$. Lemma 1.5 in the same article tells us that $U(O)$ is closed if there is a finite extension $D/M(O)$ such that $\Aut(\CC/D)\leq U(O)$. Let us now give a reason for the existence of such an extension $D$: As $E$ and the branch locus $\{\infty\}$ are defined over $\QQ$, it follows from \cite[Theorem 4.1]{gd} that $p\colon X\to E$ can be defined over a number field. Choose such a field of definition $D$, which is hence a finite extension of $M(O)$. Obviously, any element $\sigma\in\Aut(\CC)$ that fixes $D$ lies in $U(O)$, so we can apply Köck's Lemma 1.5 and finally deduce the first half of b).

Now we restate these arguments for the second equality: We use a) to deduce $[M(C(O)):\QQ]<\infty$. Furthermore, from Proposition \ref{prop:ori-curves-arithmetic} a) follows that the embedded origami curve $C(O)$ can be defined over a number field, so we can use the same chain of arguments as above.
\end{proof}

It is interesting to know which properties of origamis are Galois invariants. Let us list some fairly obvious ones:
\begin{satz}\label{satz:ori-galois-invariants}
The following properties of an origami $O=(p\colon X^*\to E^*)$ are Galois invariants:
\begin{aenum}
\item The index of the Veech group $[\SL_2(\ZZ):\Gamma(O)]$.
\item The index of the projective Veech group $[\PSL_2(\ZZ):\text{P}\Gamma(O)]$.
\item The property whether or not $-I\in\Gamma(O)$.
\item The isomorphism type of the group of translations, $\text{Trans}(O)$.
\end{aenum}
\end{satz}

This has surely been noticed before, but in lack of a better reference we refer to \cite[Section 3.7]{diss} for a proof.

As an application, we can bound the degree of the field extension $M(C(O))/M(O)$ from above in a way that looks surprising at first:

\begin{satz}\label{satz:ori-moduli-degree}
Let $O=(p\colon X^*\to E^*)$ be an origami. Then we have
\[[M(O):M(C(O))]\leq[\SL_2(\ZZ):\Gamma(O)].\]
\end{satz}

\begin{proof}
Let $O\cdot\Aut(\CC)=\{O_1,\,\ldots,\,O_k\}$ and $C(O)\cdot\Aut(\CC)=\{C_1,\,\ldots,\,C_l\}$. Then we have, as we have shown in Remark \ref{bem:ori-moduli} b):
\[k=[M(O):\QQ],\:\:l=[M(C(O)):\QQ].\]
By Theorem \ref{satz:ori-galois-invariants} the Veech groups of all $O_i$ have the same index $m\coloneqq [\SL_2(\ZZ):\Gamma(O)]$, and by Corollary \ref{cor:ori-curve-action} we also have 
\[\{C_1,\,\ldots,\,C_l\}=\{C(O_1),\,\ldots,\,C(O_k)\}.\]
From Proposition \ref{prop:ori-veechgroup} c) we know that each curve $C_j$ can be the origami curve of at most $m$ of the $O_i$'s, so we have
\[l\geq \frac{k}{m},\]
or equivalently $\tfrac{k}{l}\leq m$. The left hand side of the latter inequality is, by the multiplicity of degrees of field extensions, equal to $[M(O):M(C(O))]$, and the right hand side is by definition the index of the Veech group of $O$.
\end{proof}

\section{The Galois action on M-Origamis}\label{sec:m-ori}
The goal of this section is to recreate M.\ Möller's construction of origamis from dessins, which he used to prove the faithfulness of the Galois action on origami curves in \cite{mm1}, in a more topological way. That is, we will calculate the monodromy of these origamis (called M-Origamis here) as well as their cylinder decompositions and Veech groups, and we will reprove said faithfulness result in a way explicit enough to give examples of non-trivial Galois orbits.

\subsection{Pillow case origamis and M-Origamis}
Remember the elliptic curve of our choice $E=\CC/\ZZ^2$ and that it is defined over $\QQ$. It carries a group structure; denote its neutral by $\infty\coloneqq 0+\ZZ^2$. Further denote by $[2]\colon E\to E$ the multiplication by $2$, and by $h\colon E\to\PeC$ the quotient map under the induced elliptic involution $z+\ZZ^2\mapsto -z+\ZZ^2$. The map $h$ shall be chosen such that $h(\infty)=\infty$, and that the other critical values are $0,\,1$ and some $\lambda\in \CC$. By abuse of notation, call their preimages also $0,\,1$ and $\lambda$. By even more abuse of notation, call $\{0,\,1\,\lambda,\infty\}\subset E$ the set of Weierstraß points and recall that it is the kernel of $[2]$. Note that $[2]$ is an unramified covering of degree $4$, and $h$ is a degree $2$ covering ramified over $\{0,\,1,\,\lambda,\infty\}$.

Now let $\gamma\colon Y\to\PeC$ be a connected pillow case covering. That is, let $Y$ be a non-singular connected projective curve over $\CC$, and $\gamma$ a non-constant morphism with $\crit\gamma\subseteq\{0,1,\lambda,\infty\}$. Let $\tilde{\pi}\colon \tilde{X}\coloneqq E\times_{\PeC}Y\to E$ be its pullback by $h$, and $\delta\colon X\to\tilde{X}$ be the desingularisation of $\tilde{X}$ (as $\tilde{X}$ will have singularities in general---we will discuss them later in this section). Finally, consider the map let $[2]\circ \pi\colon X\to E$. In order not to get lost in all these morphisms, we draw a commutative diagram of the situation: 

\begin{center}
\begin{tikzpicture}[description/.style={fill=white,inner sep=2pt}]
\matrix (m) [matrix of math nodes, row sep=1.5em,
column sep=1.25em, text height=1.5ex, text depth=0.25ex]
{ X & \tilde{X} && Y\\
&&\square\\
& E && \PeC\\
\\
&E\\ };
\path[->,font=\scriptsize]
(m-1-1) edge node[auto] {$ \delta $} (m-1-2)
        edge node[left] {$ \pi $} (m-3-2)
(m-1-2) edge node[auto] {$ $} (m-1-4)
        edge node[auto] {$\tilde{\pi}$} (m-3-2)
(m-1-4) edge node[auto] {$\gamma$} (m-3-4)
(m-3-2) edge node[auto] {$h$} (m-3-4)
        edge node[auto] {$[2]$} (m-5-2);

\end{tikzpicture}
\end{center}
Because $\gamma$ is a pillow case covering, $\tilde{\pi}$ is branched (at most) over the Weierstraß points, and so is $\pi$. So, $[2]\circ \pi$ is branched over at most one points, which means that it defines an origami if $X$ is connected (which we will show in Remark \ref{rem:m-ori-connected}). 

\begin{defi}\label{defi:m-ori}
\begin{aenum}
\item In the situation above, we call $O(\gamma)\coloneqq ([2]\circ\pi:X\to E)$ the \textit{pillow case origami} associated to the pillow case covering $\gamma\colon Y\to\PeC$.
\item If furthermore $\beta\coloneqq \gamma$ is unramified over $\lambda$, i.e.\ it is a Belyi morphism, then we call $O(\beta)$ the \textit{M-Origami} associated to $\gamma$.
\end{aenum}
\end{defi}

\subsection{The topological viewpoint}
We would like to replace the algebro-geometric fibre product by a topological one, in order to apply the results derived in the first section. In fact, after taking out all critical and ramification points, it will turn out that $X^\ast \cong E^\ast \times_{\PeC^\ast}Y^\ast$ as coverings. To show this, we have to control the singular locus of $\tilde X$ in the above diagram. This is a rather elementary calculation in algebraic geometry done in the follwing

\begin{prop}\label{prop:singularities}
Let $C, D$ be non-singular projective curves over $k$, where $k$ is an algebraically closed field, and $\Phi\colon  C\to \Pek,\,\Psi\colon  D\to \Pek$ non-constant rational morphisms (i.e.\ ramified coverings). Then we have for the singular locus of $C\times_{\Pek} D$:
\[\text{Sing}(C\times_{\Pek} D)=\{(P,Q)\in C\times_{\Pek} D\mid \Phi \text{ is ramified at } P\wedge\Psi \text{ is ramified at } Q\}\]
\end{prop}
\begin{proof}
Since the property of a point of a variety to be singular can be decided locally, we can first pass on to an affine situation and then conclude by a calculation using the Jacobi criterion.

So let $(P,Q)\in F\coloneqq C\times_{\Pek} D$, and let $U\subset\Pek$ be an affine neighbourhood of $\Phi(P)=\Psi(Q)$. Further let $P\in U'\subset C,\,Q\in U''\subset D$ be affine neighbourhoods such that $\Phi(U')\subseteq U\supseteq \Psi(U'')$. This is possible since the Zariski topology of any variety admits a basis consisting of affine subvarieties (e.g.\ see \cite[I, Prop. 4.3]{har}). Now let $p_C\colon F\to C,\, p_D\colon F\to D$ be the canonical projections. Then, by the proof of \cite[II, Thm. 3.3]{har} we have $p_C^{-1}(U')\cong U'\times_{\Pek}D$, and repeating that argument on the second factor gives
\[p_C^{-1}(U')\cap p_D^{-1}(U'')=({p_D}_{|p_C^{-1}(U')})^{-1}(U'')\cong U'\times_{\Pek}U''\cong U'\times_U U''.\]
The last isomorphism is due to the easy fact that in any category, a monomorphism $S\to T$ induces an isomorphism $A\times_S B\cong A\times_T B$, given that either of the two exists. So we are, as desired, in an affine situation, as the fibre product of affine varieties is affine.

Now, let $U=\Af^1_k$, and let
\[I'=(f_1,\ldots,\,f_k)\subseteq k[x_1,\ldots,\,x_n]\eqqcolon R_n,\:I''=(g_1,\ldots,\,g_l)\subseteq k[y_1,\ldots,\,y_m]\eqqcolon R_m\]
be ideals such that $U'=V(I')\subseteq \Af^n_k,\,U''=V(I'')\subseteq \Af^m_k$.
Furthermore let $\varphi\in R_n,\,\psi\in R_m$ be polynomials  representing the morphisms $\Phi$ and $\Psi$ on the affine parts $U'$ and $U''$. We denote their images in the affine coordinate rings by $\overline{\varphi}\in k[U'],\,\overline{\psi}\in k[U'']$. So we get
\[F\coloneqq U'\times_U U''=V\left(f_1,\ldots,f_k,g_1,\ldots, g_l, \varphi(x_1,\ldots, x_n)-\psi(y_1,\ldots, y_m)\right)\subseteq \Af^{m+n}.\]
The Jacobi matrix of $F$ is given by
\[
J_F=\begin{pmatrix}
\frac{\partial f_1}{\partial x_1}&\dots&\frac{\partial f_1}{\partial x_n}\\
\vdots && \vdots&&0\\
\frac{\partial f_k}{\partial x_1}&\dots&\frac{\partial f_k}{\partial x_n}\\
&&&\frac{\partial g_1}{\partial y_1}&\dots&\frac{\partial g_1}{\partial y_m}\\
&0&&\vdots && \vdots\\
&&&\frac{\partial g_l}{\partial y_1}&\dots&\frac{\partial g_l}{\partial y_m}\\
\frac{\partial\varphi}{\partial x_1}&\dots&\frac{\partial\varphi}{\partial x_n}&-\frac{\partial\psi}{\partial y_1}&\dots&-\frac{\partial\psi}{\partial y_m}
\end{pmatrix}
\]
and by the Jacobi criterion a point $(P,Q)\in F$ is singular iff $\text{rk}(J_F(P,Q))< m+n-1$.

For $P=(p_1,\ldots,\,p_n)\in \Af^n_k,\,Q=(q_1,\ldots,\,q_m)\in\Af^m_k$, denote by
\begin{align*}
M_P&\coloneqq \bigl((x_1-p_1),\ldots,(x_n-p_n)\bigr)\subseteq R_n\text{ and}\\
M_Q&\coloneqq \bigl((y_1-q_1),\ldots,(y_m-q_m)\bigr)\subseteq R_m
\end{align*}
the corresponding maximal ideals, and, if $P\in U'$ and $Q\in U''$, denote by $m_P\subseteq k[U']$ and $m_Q\subseteq k[U'']$ the corresponding maximal ideals in the affine coordinate rings. Before we continue, we note that for $h\in R_n$, we have:
\begin{equation}\label{eq:partial}
h\in M_P^2\Leftrightarrow h(P)=0\wedge\forall i=1,\ldots,\,n\colon \frac{\partial h}{\partial x_i}(P)=0.
\end{equation}
Of course the corresponding statement is true for $M_Q^2\subseteq R_m$.

Now let $P=(p_1,\ldots,p_n)\in U',\,Q=(q_1,\ldots,q_m)\in U''$ be ramification points of $\overline{\varphi}$ and  $\overline{\psi}$, respectively. This is by definition equivalent to
\[\overline{\varphi}-\overline{\varphi}(P)\in m_P^2\:\wedge\:\overline{\psi}-\overline{\psi}(Q)\in m_Q^2, \text{ or,}\]
\begin{equation}\label{eq:M^2+I}
\varphi-\varphi(P)\in M_P^2+I'\:\wedge\:\psi-\psi(Q)\in M_Q^2+I''.
\end{equation}
So, there exist $a_0\in M_P^2,\,a_1,\ldots,\,a_k\in R_n,\:b_0\in M_Q^2,\,b_1,\ldots,\,b_l\in R_m$ such that
\[\varphi-\varphi(P)=a_0+\sum_{i=1}^k a_if_i,\:\:\psi-\psi(Q)=b_0+\sum_{i=1}^l b_ig_i.\]
Writing $R_n\ni h= h(P)+(h-h(P))$, we have the decomposition $R_n=k\oplus M_P$ as $k$-modules, and analogously $R_m=k\oplus M_Q$. So write
\[a_i=\lambda_i+\tilde{a}_i\in k\oplus M_P,\,i=1,\ldots,\,k\text{ and }b_i=\mu_i+\tilde{b}_i\in k\oplus M_Q,\,i=1,\ldots,\,l.\]
Because $I'\subseteq M_P$ and $I''\subseteq M_Q$, we have $c\coloneqq \sum_{i=1}^k \tilde{a}_if_i\in M_P^2$ and $d\coloneqq \sum_{i=1}^l \tilde{b}_ig_i\in M_Q^2$. So if we set $a\coloneqq a_0+c\in M_P^2$ and $b\coloneqq b_0+d\in M_Q^2$, we get
\[\varphi-\varphi(P)=a+\sum_{i=1}^k\lambda_if_i\text{ and }\psi-\psi(Q)=b+\sum_{i=1}^l\mu_ig_i.\]
Deriving on both sides of the equations with respect to all the variables, we get, using (\ref{eq:partial}):
\begin{align*}\forall i\in\{1,\ldots,\,n\}\colon & \frac{\partial\varphi}{\partial x_i}(P)=\sum_{j=1}^k\lambda_j\frac{\partial f_j}{\partial x_i}(P) \text{ and }\\
  \forall i\in\{1,\ldots,\,m\}\colon & \frac{\partial\psi}{\partial y_i}(Q)=\sum_{j=1}^l\mu_j\frac{\partial g_j}{\partial y_i}(Q).\end{align*}

So the last row of $J_F(P,Q)$ is a linear combination of the first $m+n$ ones.
As the two big non-zero blocks of $J_F$ are simply the Jacobi matrices $J_{U'}$ and $J_{U''}$ of the non-singular curves $U'$ and $U''$, they have ranks $n-1$ and $m-1$, evaluated at $P$ and $Q$ respectively, and we get
\[\text{rk}\left(J_F(P,Q)\right)=\text{rk}(J_{U'}(P))+\text{rk}(J_{U''}(Q))=(n-1)+(m-1)<m+n-1,\]
so $(P,Q)\in\text{Sing}(U'\times_U U'')$.

Conversely, let $(P,Q)\in\text{Sing}(F)$ be a singular point. Then, by the non-singularity of $U'$ and $U''$, we have $\text{rk}(J_F(P,Q))=m+n-2$. More specifically, the last row of this matrix is a linear combination of the others:
\begin{align*}
\exists \lambda_1,\ldots, \lambda_k\in k\:\:\forall i\in\{1,\ldots,\,n\}\colon & \frac{\partial\varphi}{\partial x_i}(P)=\sum_{j=1}^k\lambda_j\frac{\partial f_j}{\partial x_i}(P) \text{ and }\\
  \exists \mu_1,\ldots, \mu_l\in k\:\:\forall i\in\{1,\ldots,\,m\}\colon & \frac{\partial\psi}{\partial y_i}(Q)=\sum_{j=1}^l\mu_j\frac{\partial g_j}{\partial y_i}(Q).
\end{align*}
Now set $\tilde{\varphi}\coloneqq \varphi-\varphi(P)-\sum\lambda_j f_j$ and $\tilde{\psi}\coloneqq \psi-\psi(Q)-\sum\mu_j g_j$. Then we have $\tilde{\varphi}\in M_P$ and $\tilde{\psi}\in M_Q$, and furthermore
\[\forall i\in\{1,\ldots,\,n\}\colon  \frac{\partial\tilde{\varphi}}{\partial x_i}(P)=0\text{ and }
\forall i\in\{1,\ldots,\,m\}\colon  \frac{\partial\tilde{\psi}}{\partial y_i}(Q)=0.\]
So we can apply (\ref{eq:partial}) again to get $\tilde{\varphi}\in M_P^2,\:\tilde{\psi}\in M_Q^2$, or,
\[\varphi-\varphi(P)\in M_P^2+I'\:\wedge\:\psi-\psi(Q)\in M_Q^2+I''.\]
This is precisely (\ref{eq:M^2+I}), which we have already shown above to be equivalent to $P$ and $Q$ being ramification points of $\overline{\varphi}$ and $\overline{\psi}$, respectively.\qed
\end{proof}

So if we denote $\Pe^*\coloneqq \PeC\setminus\{0,\,1,\,\lambda,\, \infty\}$ and by $E^*,\, Y^*,\, X^*$ its preimages under $h,\, \beta$ and $\pi\circ h$, respectively, (and, for the sake of a simpler notation, do not change the names of the restricted maps,) then we have in particular:

\begin{bem}
The following diagram is Cartesian in the category of topological spaces together with covering maps:
\begin{center}
\begin{tikzpicture}[description/.style={fill=white,inner sep=2pt}]
\matrix (m) [matrix of math nodes, row sep=1.5em,
column sep=1.25em, text height=1.5ex, text depth=0.25ex]
{ X^* & & Y^*\\
&\square\\
E^* && \Pe^*\\
};
\path[->,font=\scriptsize]
(m-1-1) edge node[auto] {$ $} (m-1-3)
        edge node[auto] {$\pi$} (m-3-1)
(m-1-3) edge node[auto] {$\beta$} (m-3-3)
(m-3-1) edge node[auto] {$h$} (m-3-3);

\end{tikzpicture}
\end{center}
\end{bem}

\subsection{The monodromy of M-Origamis}

We can now calculate the monodromy of an M-Origami $M(\beta)$, given the monodromy of a Belyi morphism $\beta$ by applying Propositions \ref{satz:fprod} and \ref{satz:comp}. Fix generators $\langle A,B\rangle=\pi_1(E\setminus\{\infty\})$ and $\langle x,y\rangle=\pi_1(\PeC\setminus\{0,1,\infty\})$ according to the figures of the following proof. We will, again, make use of the “two-coordinate” way of labeling fibres of composed coverings we introduced in Proposition \ref{satz:comp}.

\begin{satz}\label{satz:moncalc}
Let $\beta\colon Y\to\PeC$ be a Belyi morphism of degree $d$ with monodromy given by $p_x\coloneqq m_\beta(x),\,p_y\coloneqq m_\beta(y)$. Then, the monodromy of $M(\beta)=([2]\circ\pi\colon X\to E)$ is given by
\begin{gather*}
m_{[2]\circ\pi}\colon \pi_1(E\setminus\{\infty\})\to S_{4d}\\ \\
m_{[2]\circ\pi}(A)(i,j)=\begin{cases}
(2,j)& i=1\\
(1, p_y(j))& i=2\\
(4,j)& i=3\\
(3,p_y^{-1}(j))& i=4\end{cases}
,\quad
m_{[2]\circ\pi}(B)(i,j)=\begin{cases}
(3,j)& i=1\\
(4,j)& i=2\\
(1,p_x^{-1}(j))& i=3\\
(2,p_x(j))& i=4\end{cases}.
\end{gather*}
\end{satz}

\begin{proof}
We imagine $\Pe^*$ as a \textit{pillow case}, i.e.\ an euclidean rectangle $R$ with width-to-height ratio $2:1$, folded in half and stitched together in the obvious way. The two to one covering $h$ is then realised by taking two copies of $R$, rotating one of them by an angle of $\pi$, stitching them together and then identifying opposite edges in the usual way. See Figure \ref{fig:pillowcase} for a sketch of this situation.
\begin{figure}[h]
\begin{center}
\includegraphics[scale=0.8]{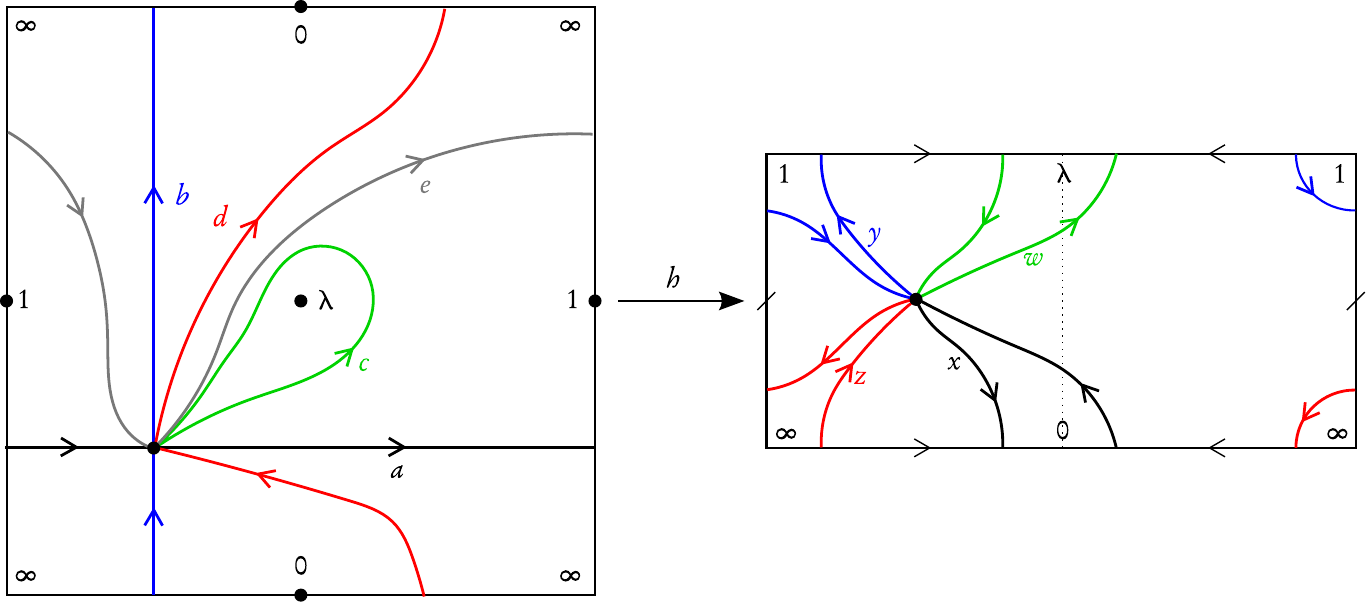}
\caption{}\label{fig:pillowcase}
\end{center}
\end{figure}
We have $\pi_1(\Pe^*)\cong F_3$. As depicted, take the paths $w,\,x,\,y,\,z$ as generators, subject to the relation $wxzy=1$. Now, we have $\pi_1(E\setminus\{0,\,1,\,\lambda,\,\infty\})\cong F_5$, and as in the picture we choose $a,\,b,\,c,\,d,\,e$ as free generators. We can verify easily that
\[h_*\colon \begin{cases} \pi_1(E\setminus\{0,\,1,\,\lambda,\,\infty\})\to \pi_1(\Pe^*),\\
a  \mapsto yw,\,
b  \mapsto x^{-1}w^{-1},\,
c  \mapsto w^2,\,
d  \mapsto xw^{-1},\,
e  \mapsto y^{-1}w^{-1}
\end{cases}.\]
Next we calculate the monodromy of the covering $\pi\colon X\setminus\pi^{-1}(\{0,\,1,\,\lambda,\,\infty\})\to E\setminus\{0,\,1,\,\lambda,\,\infty\}$.
As $\beta$ is unramified over $\lambda$, we have $m_\beta(w)=1$ and $m_\beta(z)=p_x^{-1}p_y^{-1}$. Applying Proposition \ref{satz:fprod} a) we get
\[m_\pi\colon \begin{cases}\pi_1(E\setminus\{0,\,1,\,\lambda,\,\infty\})\to S_d,\\
a \mapsto p_y,\,
b \mapsto p_x^{-1},\,
c \mapsto 1,\,
d \mapsto p_x,\,
e \mapsto p_y^{-1}
\end{cases}.\]

Now, we apply Proposition \ref{satz:comp} to calculate the monodromy of $[2]\circ\pi$. Choose a base point $x_0$ and paths $A$ and $B$ as generators of $\pi_1(E\setminus\{\infty\})$, and label the elements of $[2]^{-1}(x_0)$ by $y_1,\ldots,y_4$, as indicated in Figure \ref{fig:mult2}. With the sketched choice of numbering, we get that the monodromy of $[2]$ is given by
\[m_{[2]}(A)=(1\,2)(3\,4),\,\,m_{[2]}(B)=(1\,3)(2\,4).\]
\begin{figure}[h]
\begin{center}
\includegraphics[scale=0.8]{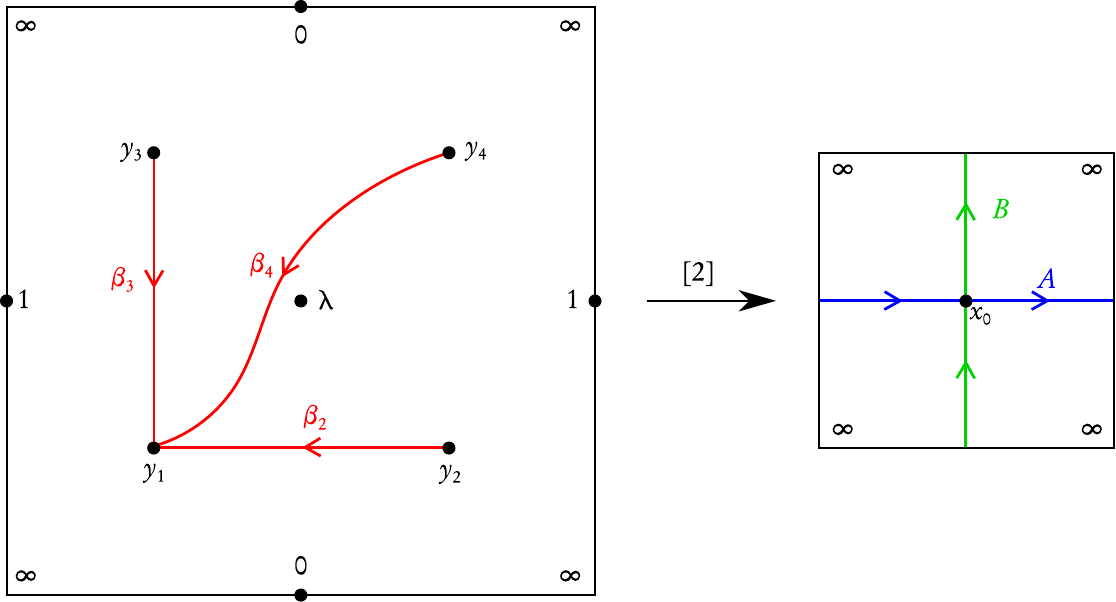}
\caption{}\label{fig:mult2}
\end{center}
\end{figure}
A right coset representatives of $m_{[2]}^{-1}(\Stab(1))$ in $\pi_1(E\setminus\{\infty\})$, we choose
\[\gamma_1\coloneqq 1,\,\gamma_2\coloneqq A^{-1},\,\gamma_3\coloneqq B^{-1},\,\gamma_4\coloneqq B^{-1}A^{-1}\]
and check that they do what they should by seeing that they lift along $[2]$ to paths $\beta_i$ connecting $y_i$ to $y_1$.
The next step is to calculate the $c_i$'s. We get
\[
\begin{split}
&c_i(A)=[2]_*^{-1}(\gamma_kA\gamma_i^{-1})=
\begin{cases}
[2]_*^{-1}(A^{-1}A\cdot 1)=1,    & i=1\\
[2]_*^{-1}(1\cdot AA)=a,         & i=2\\
[2]_*^{-1}(B^{-1}A^{-1}AB)=1,    & i=3\\
[2]_*^{-1}(B^{-1}AAB)=e,         & i=4
\end{cases}\quad\text{and}\\
&c_i(B)=[2]_*^{-1}(\gamma_kB\gamma_i^{-1})=
\begin{cases}
[2]_*^{-1}(B^{-1}B\cdot 1)=1,    & i=1\\
[2]_*^{-1}(B^{-1}A^{-1}BA)=c,    & i=2\\
[2]_*^{-1}(1\cdot BB)=b,         & i=3\\
[2]_*^{-1}(A^{-1}BAB)=d,         & i=4
\end{cases}.
\end{split}
\]
Putting all together, we get the claimed result
\[
\begin{split}
&m_{[2]\circ\pi}(A)(i,j)=\left(m_{[2]}(A)(i),\,m_\pi(c_i(A))(j)\right)=
\begin{cases}
(2,j)           & i=1\\
(1, p_y(j))     & i=2\\
(4,j)           & i=3\\
(3,p_y^{-1}(j)) & i=4
\end{cases}\quad\text{and}\\
&m_{[2]\circ\pi}(B)(i,j)=\left(m_{[2]}(B)(i),\,m_\pi(c_i(B))(j)\right)=
\begin{cases}
(3,j)           & i=1\\
(4,j)           & i=2\\
(1,p_x^{-1}(j)) & i=3\\
(2,p_x(j))      & i=4
\end{cases}.
\end{split}
\]
\end{proof}
Now we can easily fill in the gap left open in the beginning of the previous section:
\begin{bem}\label{rem:m-ori-connected} If we start with a Belyi morphism $\beta$, the topological space $X^*$ arising in the construction is always connected, so $O(\beta)$ is indeed an origami.
\end{bem}
\begin{proof}
What we have to show is that $m_A\coloneqq m_{[2]\circ\pi}(A)$ and $m_B\coloneqq m_{[2]\circ\pi}(B)$ generate a transitive subgroup of $S_{4d}$. So choose $i\in\{1,\ldots, 4\},\,j\in\{1,\ldots,d\}$, then it clearly suffices to construct a path $\gamma\in \pi_1(E\setminus\{\infty\})$ with the property that $m_{[2]\circ\pi}(\gamma)(1,1)=(i,j)$.

Now, as $Y^*$ is connected, there is a path $\gamma'\in\pi_1(\Pe^*)$ such that $m_\beta(\gamma')(1)=j$. Consider the homomorphism
\[\phi\colon \pi_1(\Pe^*)\to\pi_1(E\setminus\{\infty\}),\,\begin{cases} x \mapsto &B^{-2},\\ y \mapsto &A^2 \end{cases} \]
By the above theorem, a lift of the path $A^2$ connects the square of $O(\beta)$ labelled with $(1,k)$ to the one labelled with$(1,p_y(k))$, and a lift of $B^{-2}$ connects $(1,k)$ to $(1,p_x(k))$. So we get $m_{[2]\circ\pi}(\phi(\gamma'))(1,1)=(1,j)$. So to conclude, we set $\gamma\coloneqq \epsilon\phi(\gamma')$, where $\epsilon\coloneqq 1,\,A,\, B$ or  $AB$ for $i=1,\,2,\,3$ or $4$, respectively, and get $m_{[2]\circ\pi}(\gamma)(1,1)=(i,j)$.
\end{proof}

Later we will want to know the monodromy of the map $\pi$ around the Weierstraß points $0,\,1,\,\lambda,\,\infty\in E$, respectively. Therefore, we quickly read off of Figure \ref{fig:pillowcase}:
\begin{hilfssatz}\label{hilfssatz:mon-around-weierstrass}
If we choose the following simple loops around the Weierstraß points
\begin{align*}
x' \coloneqq  l_0 &\coloneqq  db^{-1} & y' \coloneqq  l_1 &\coloneqq  ac^{-1}e^{-1}\\
z' \coloneqq  l_\infty &\coloneqq  bed^{-1}a^{-1} & w' \coloneqq  l_\lambda &\coloneqq  c,
\end{align*}
then we have
\begin{align*}
h_*(x') &= x^2 & h_*(y') &= y^2\\
h_*(z') &= z^2 & h_*(w') &= w^2.
\end{align*}
\end{hilfssatz}

\subsection{The genus and punctures of M-Origamis}
We calculate  the genus of the M-Origami associated to a Belyi morphism $\beta$ and give lower and upper bounds depending only on $g(Y)$ and $\deg \beta$. We have the following
\begin{prop}\label{prop:genus}
Let $\beta\colon Y\to\PeC$ be a Belyi morphism and $d=\deg\beta$ its degree, $m_0\coloneqq p_x,\, m_1\coloneqq p_y,\, m_\infty\coloneqq p_z=p_x^{-1}p_y^{-1}\in S_d$ the monodromy of the standard loops around $0,\,1$ and $\infty$. Further let $g_0,\,g_1,\,g_\infty$ be the number of cycles of even length in a disjoint cycle decomposition of $m_0,\, m_1$ and $m_\infty$, respectively. Then we have:
\begin{enumerate}
\item[a)] \[g(X)=g(Y)+d-\frac{1}{2}\sum_{i=0}^\infty g_i.\]
\item[b)] \[g(Y)+\left\lceil\frac{d}{4}\right\rceil\leq g(X) \leq g(Y)+d.\]
\end{enumerate}
\end{prop}
\begin{proof}
The Riemann-Hurwitz formula for $\beta$ says:
\[2g(Y)-2=d(2g(\Pe)-2)+\sum_{p\in Y}(e_p-1)\]
Let us now split up $\sum_{p\in Y}(e_p-1)=v_0+v_1+v_\infty$, where $v_i\coloneqq \sum_{p\in \beta^{-1}(i)}(e_p-1)$. Of course, $v_i=\sum_{c\text{ cycle in } m_i}(\len(c)-1)$, where $\len(c)$ is the length of a cycle $c$. Putting that, and the fact that $g(\Pe)=0$, into the last equation yields
\[2g(Y)-2=-2d+\sum_{i=0}^\infty v_i.\]
Now we write down the Riemann-Hurwitz formula for $\pi$. Note that $\pi$ is at most ramified over the preimages of $0,\,1$ and $\infty$ under $h$, which we denote also by $0,\,1$ and $\infty$. $\beta$ is unramified over the (image of the) fourth Weierstraß point, and so is $\pi$. So again the ramification term splits up into $\sum_{p\in X}(e_p-1)=v'_0+v'_1+v'_\infty$ with $v'_i=\sum_{c\text{ cycle in } m'_i}(\len(c)-1)$, where $m'_0, m'_1, m'_\infty$ are permutations describing the monodromy of $\pi$ going around the Weierstraß points $0, 1, \infty$ respectively. Of course $g(E)=1$, and so we get
\[2g(X)-2=\sum_{i=0}^\infty v'_i.\]
Subtracting from that equation the one above and dividing by $2$ yields
\[g(X)=g(Y)+d-\frac{1}{2}\sum_{i=0}^\infty(v_i-v'_i),\]
so we can conclude a) with the following
\begin{hilfssatz} $v_i-v'_i=g_i$, where $g_i$ is the number of cycles of even length in $m_i$.
\end{hilfssatz}
The proof of this lemma is elementary. Write $m'_0\coloneqq m_{\pi}(x'),\,m'_1\coloneqq m_{\pi}(y'),\allowbreak\,m'_\infty\coloneqq m_{\pi}(z')$ as in Lemma \ref{hilfssatz:mon-around-weierstrass}, then we get $m'_i=m_i^2$ by that lemma and Proposition \ref{satz:fprod}. Now, note that the square of a cycle of odd length yields a cycle of the same length, while the square of a cycle of even length is the product of two disjoint cycles of half length. So, if $c$ is such an even length cycle and $c^2=c_1c_2$, then obviously $(\len(c)-1)-((\len(c_1)-1)+(\len(c_2)-1))=1$. Summing up proves the lemma.\qed

In part b) of the proposition, the second inequality is obvious as all $g_i$ are non-negative. For the first one, note that we have $g_i\leq \frac{d}{2}$, so $\sum g_i\leq \frac{3}{2}d$, and finally $d-\frac{1}{2}\sum g_i\geq \frac{d}{4}$. Of course $g(X)$ and $g(Y)$ are natural numbers, so we can take the ceiling function.
\end{proof}

It can be explicitly shown that the upper bound of part b) is sharp in some sense: In \cite[Remark 4.9]{diss}, for each $g\in\NN$ and for $d$ large enough, we construct dessins of genus $g$ and degree $d$, such that the resulting M-Origami has genus $g+d$.

A direct consequence of Proposition \ref{prop:genus}, and the fact that there are only finitely many dessins up to a given degree, is the following observation: 

\begin{kor} \label{kor:finite} Given a natural number $G\in \NN$, there are only finitely many M-Origamis with genus less or equal to $G$.
\end{kor}

For example, there are five dessins giving M-Origamis of genus $1$ and nine giving M-Origamis of genus $2$.

With the notations of the above proposition and the considerations in the proof, we can easily count the number of punctures of an M-Origami:

\begin{bem}\label{bem:m-ori-punctures}
Let $\beta\colon Y\to\PeC$ be a Belyi morphism, $\pi\colon  X\to E$ the normalisation of its pullback by the elliptic involution as in Definition \ref{defi:m-ori}, and $O_\beta=(p\colon X\to E)$ the associated M-Origami.
\begin{aenum}
\item Let $W\coloneqq \{0,\,1,\,\lambda,\,\infty\}\subseteq E$ be the set of Weierstraß points of $E$, then we have:
\begin{align*}
n(P)&\coloneqq |\pi^{-1}(P)|= \#(\text{cycles in }m_P)+g_P,\:\:(P\in\{0,\,1,\,\infty\})\\
n(\lambda)&\coloneqq |\pi^{-1}(\lambda)|=\deg \pi=\deg\beta.\\
\end{align*}
\item For $O_\beta=(p\colon X\to E)$, we have:
\[|p^{-1}(\infty)|=n(0)+n(1)+n(\lambda)+n(\infty).\]
\end{aenum}
\end{bem}
\begin{proof}
Part b) is a direct consequence of a), as by definition $p=[2]\circ\pi$, and $[2]\colon E\to E$ is an unramified covering with $[2]^{-1}(\infty)=W$.

For a), the second equality is clear, since, as we noted above, $\pi$ is unramified over $\lambda$. If $P\in\{0,\,1\,\infty\}$, then $n(P)$ is the number of cycles in $m_P'=m_P^2$ (where we reuse the notation from the proof above). As a cycle in $m_P$ corresponds to one cycle in $m_P^2$ if its length is odd, and splits up into two cycles of half length if its length is even, we get the desired statement.
\end{proof}

\subsection{The Veech group}
We can now calculate the Veech group of an M-Origami with the help of Theorem \ref{satz:gabi-vg}. We begin by calculating how $\SL_2(\ZZ)$ acts:
\begin{prop}\label{prop:STaction}
Let $\beta$ be the dessin of degree $d$ given by a pair of permutations $(p_x,\,p_y)$, and $O_\beta$ the associated M-Origami. Then we have for the standard generators $S,\,T,\,-I$ of $\SL_2(\ZZ)$:
\begin{itemize}
\item $S\cdot O_\beta$ is the M-Origami associated to the pair of permutations $(p_y,\,p_x)$.
\item $T\cdot O_\beta$ is the M-Origami associated to the pair of permutations $(p_z,\,p_y)$, where as usual $p_z=p_x^{-1}p_y^{-1}$.
\item $(-I)\cdot O_\beta\cong O_\beta$, i.e.\ $-I\in\Gamma(O_\beta)$.
\end{itemize}
\end{prop}
\begin{proof}
We lift  $S,\,T,\,-I\in \SL_2(\ZZ)\cong \text{Out}^+(F_2)$ to the following automorphisms of $F_2=\langle A,\,B\rangle$, respectively:
\begin{align*}
\phi_S\colon & A\mapsto B,\,B \mapsto A^{-1}\\
\phi_T\colon & A\mapsto A,\,B\mapsto AB\\
\phi_{-I}\colon & A\mapsto A^{-1},\,B\mapsto B^{-1}
\end{align*}
Let us, for brevity, write $m_w\coloneqq m_{[2]\circ\pi}(w)$ for an element $w\in F_2$, so in particular $O_\beta$ is given by the pair of permutations $(m_A,m_B)$ as in Theorem \ref{satz:moncalc}.

First we calculate the monodromy of the origami $S\cdot O_\beta$, which is given by $(m_{\phi_S^{-1}(A)},m_{\phi_S^{-1}(B)})$: We get
\[m_{\phi_S^{-1}(A)}(i,j)=m_B^{-1}(i,j)=\begin{cases}
                    (3,p_x(j)),&i=1\\
                    (4,p_x^{-1}(j)),&i=2\\
                    (1,j),&i=3\\
                    (2,j),&i=4\\
                \end{cases}\text{ and }m_{\phi_S^{-1}(B)}=m_A.\]
Now we conjugate this pair by the following permutation 
\[c_1\colon (1,\,j)\mapsto(2,\,j)\mapsto(4,\,j)\mapsto(3,\,j)\mapsto (1,\,j),\]
which does of course not change the origami it defines, and indeed we get:
\[c_1m_B^{-1}c_1^{-1}(i,j)=\begin{cases}
                    (2,j),&i=1\\
                    (1,p_x(j)),&i=2\\
                    (4,j),&i=3\\
                    (3,p_x^{-1}(j)),&i=4\\
                \end{cases},\,
c_1m_Ac_1^{-1}(i,j)=\begin{cases}
                    (3,j),&i=1\\
                    (4,j),&i=2\\
                    (1,p_y^{-1}(j)),&i=3\\
                    (2,p_y(j)),&i=4\\
                \end{cases},\]
which is clearly the M-Origami associated to the dessin given by the pair $(p_y,\,p_x)$.

Next, we discuss the action of the element $T$ in the same manner, and we get:
\[m_{\phi_T^{-1}(A)}=m_A\text{ and }
m_{\phi_T^{-1}(B)}(i,j)=m_A^{-1}m_B(i,j)=\begin{cases}
                    (4,p_y(j)),&i=1\\
                    (3,j),&i=2\\
                    (2,p_y^{-1}p_x^{-1}(j)),&i=3\\
                    (1,p_x(j)),&i=4
                \end{cases}.\]
The reader is invited to follow the author in not losing hope and verifying that with
\[c_2\colon (1,j)\mapsto (1,p_y(j)),\,
            (2,j)\mapsto (2,p_y(j)),\,
            (3,j)\mapsto (4,p_y(j)),\,
            (4,j)\mapsto (3,j)\]
we have
\[
c_2m_Ac_2^{-1}(i,j)=\begin{cases}
                    (2,j),&i=1\\
                    (1,p_y(j)),&i=2\\
                    (4,j),&i=3\\
                    (3,p_y^{-1}(j)),&i=4\\
                \end{cases},\,
c_2m_A^{-1}m_Bc_2^{-1}(i,j)=\begin{cases}
                    (3,j),&i=1\\
                    (4,j),&i=2\\
                    (1,p_yp_x(j)),&i=3\\
                    (2,p_x^{-1}p_y^{-1}(j)),&i=4\\
                \end{cases},
\]
which is the monodromy of the M-Origami associated to $(p_z,\,p_y)$.

For the element $-I\in \SL_2(\ZZ)$ we have to calculate $m_{\phi_{-I}^{-1}(A)}=m_A^{-1}$ and $m_{\phi_{-I}^{-1}(B)}=m_B^{-1}$ which evaluate as
\[ m_A^{-1}(i,j)=\begin{cases}
                    (2,p_y^{-1}(j)),&i=1\\
                    (1,j),&i=2\\
                    (4,p_y(j)),&i=3\\
                    (3,j),&i=4\\
                \end{cases},\,
  m_B^{-1}(i,j)=\begin{cases}
                    (3,p_x(j)),&i=1\\
                    (4,p_x^{-1}(j)),&i=2\\
                    (1,j),&i=3\\
                    (2,j),&i=4\\
                \end{cases}.
\]
In this case, the permutation
\[c_3\colon (1,j)\leftrightarrow (4,j),\, (2,j)\leftrightarrow (3,j)\]
does the trick and we verify that $c_3m_Ac_3=m_A^{-1},\,c_3m_Bc_3=m_B^{-1}$, so $-I\in\Gamma(O_\beta)$.
\end{proof}
It will turn out in the next theorem that the Veech group of M-Origamis often is $\Gamma(2)$. Therefore we list, as an easy corollary from the above proposition, the action of a set of coset representatives of $\Gamma(2)$ in $\SL_2(\ZZ)$.
\begin{kor}\label{kor:sl2-action}
For an M-Origami $O_\beta$ associated to a dessin $\beta$ given by $(p_x,\,p_y)$, and 
\[\alpha\in\{I,\,S,\,T,\,ST,\,TS,\,TST\},\]
$\alpha\cdot O_\beta$ is again an M-Origami, and it is associated to the dessin with the monodromy indicated in the following table:
\begin{center}
	\begin{tabular}{|c|c|c|c|c|c|}
	    \hline
		$I$ &   $S$ &   $T$ &   $ST$ &  $TS$ &  $TST$ \\\hline\hline
		$p_x$ & $p_y$ & $p_z$ & $p_y$ & $p_z$ & $p_x$ \\\hline
		$p_y$ & $p_x$ & $p_y$ & $p_z$ & $p_x$ & $p_z$ \\\hline
	\end{tabular}
\end{center}
\end{kor}
\begin{proof}
The first three columns of the above table are true by the above proposition. If we write $M(p_x,\,p_y)$ for the M-Origami associated to the dessin given by $(p_x,\,p_y)$, then we calculate
\begin{gather*}
ST\cdot M(p_x,\,p_y)=S\cdot M(p_z,\,p_y)=M(p_y,\,p_z),\\
TS\cdot M(p_x,\,p_y)=T\cdot M(p_y,\,p_x)=M(p_y^{-1}p_x^{-1},p_x)\cong M(p_z,\,p_x),\\
TST\cdot M(p_x,\,p_y)=T\cdot M(p_y,p_z)=M(p_y^{-1}p_z^{-1},p_z)=M(p_x,\,p_z).
\end{gather*}
\end{proof}
\begin{satz}\label{satz:veechgroup}
Let, again, $\beta$ be a dessin given by the pair of permutations $(p_x,\,p_y)$.
\begin{aenum}
\item For the associated M-Origami $O_\beta$, we have $\Gamma(2)\subseteq\Gamma(O_\beta)$.
\item The orbit of $O_\beta$ under $\SL_2(\ZZ)$ precisely consists of the M-Origamis associated to the dessins weakly isomorphic to $\beta$.
\item If $\Gamma(2)=\Gamma(O_\beta)$ then $\beta$ has no non-trivial weak automorphism, i.e.\ $W_\beta=\{\id\}$. In the case that $\beta$ is filthy, the converse is also true.
\end{aenum}
\end{satz}
\begin{proof}
\begin{aenum}
\item We have $\Gamma(2)=\langle T^2,\,ST^{-2}S^{-1},\,-I\rangle$, so we have to show that these three matrices are elements of $\Gamma(O_\beta)$. We already know by Proposition \ref{prop:STaction} that $-I$ acts trivially. Using the notation of the proof of the above corollary, we calculate:
\begin{gather*}
T^2\cdot M(p_x,\,p_y)=T\cdot M(p_z,\,p_y)=M(p_z^{-1}p_y^{-1},\,p_y)\cong M(p_x,\,p_y),\\
ST^{-2}S^{-1}\cdot M(p_x,\,p_y)=ST^{-2}\cdot M(p_y,\,p_x)=S\cdot M(p_y,\,p_x) = M(p_x,\,p_y).
\end{gather*}
\item By a), the orbit of $O_\beta$ under $\SL_2(\ZZ)$ is the set of translates of $O_\beta$ under a set of coset representatives of $\Gamma(2)$ in $\SL_2(\ZZ)$. We have calculated them in the above corollary, and indeed they are associated to the dessins weakly isomorphic to $\beta$.
\item “$\Rightarrow$”: Let $\Gamma(2)=\Gamma(O_\beta)$, then by part b) we have
\[6=|\SL_2(\ZZ)\cdot O_\beta|\leq |W\cdot \beta|\leq 6\]
so we have equality and indeed $W_\beta$ is trivial.\\
“$\Leftarrow$”: For now, fix an element $\id\neq w\in W$. By assumption, $\beta\ncong\beta'\coloneqq w\cdot\beta$. Let $\pi,\,\pi'$ be their pullbacks by $h$ as in Definition \ref{defi:m-ori}. By Lemma \ref{hilfssatz:beta-to-pi}\footnote{Note that this and the following lemma logically depend only on the calculations in the proof of Theorem \ref{satz:moncalc}.}, we have $\pi\ncong\pi'$. Now assume $O_\beta\cong O_{\beta'}$, so by Lemma \ref{hilfssatz:deck-existence}, there exists a deck transformation $\varphi\in\Deck ([2])$ such that $\pi'\cong \varphi\circ \pi$. But since $\beta$ (and so $\beta'$) is filthy, $\varphi$ has to fix $\lambda\in E$, so it is the identity. This is the desired contradiction to $\pi\ncong\pi'$. By varying $w$ we get $|\SL_2(\ZZ)\cdot O_\beta|=6$, so in particular $\Gamma(O_\beta)=\Gamma(2)$.
\end{aenum}
\end{proof}

Part b) of the above theorem indicates a relationship between weakly isomorphic dessins and affinely equivalent M-Origamis. Let us understand this a bit more conceptually:

\begin{prop}\label{prop:equivariant-action}
The group $W$ from Definition and Remark \ref{defbem:weak-action} acts on the set of origamis whose Veech group contains $\Gamma(2)$ via the group isomorphism
\[\phi\colon W\to \SL_2(\ZZ)/\Gamma(2),\,\begin{cases}s\mapsto \overline{S}\\t\mapsto\overline{T}\end{cases}.\]
Furthermore, the map $M\colon \beta\mapsto O_\beta$, sending a dessin to the corresponding M-Origami, is $W$-equivariant.
\end{prop}

\begin{proof}
By the proof of Proposition \ref{prop:ori-veechgroup} c), the action of $\SL_2(\ZZ)$ on the set of origamis whose Veech group contains $\Gamma(2)$ factors through $\Gamma(2)$. So any group homomorphism $G\to\SL_2(\ZZ)/\Gamma(2)$ defines an action of $G$ on this set. To see that the map $M$ is equivariant with respect to the actions of $W$ on dessins and M-Origamis, respectively, amounts to comparing the tables in Definition and Remark \ref{defbem:weak-action} and Corollary \ref{kor:sl2-action}.
\end{proof}

\subsection{Cylinder decomposition}
By the results of the above section, for an M-Origami $O_\beta$ we find that $\Hp/\Gamma(2)\cong\PeC\setminus\{0,\,1,\,\infty\}$ covers its origami curve $C(O_\beta)$ which therefore has at most three cusps. By Proposition \ref{prop:ori-strebel-cylinders} a) this means that $O_\beta$ has at most three non-equivalent Strebel directions, namely $\left(\begin{smallmatrix}1\\0\end{smallmatrix}\right),\,\left(\begin{smallmatrix}0\\1\end{smallmatrix}\right)$ and $\left(\begin{smallmatrix}1\\1\end{smallmatrix}\right)$. For each of these, we will calculate the cylinder decomposition. Before, let us prove the following lemma which will help us assert a peculiar condition appearing in the calculation of the decomposition:
\begin{hilfssatz}\label{hilfssatz:cyl-lemma}
Let $\beta$ be a Belyi morphism, and let $p_x, \,p_y,\, p_z$ be the monodromy around $0,\,1,\,\infty$ as usual. If, for one cycle of $p_y$ that we denote w.l.o.g.\ by $(1\ldots k)$, $k\leq \deg(\beta)$, we have
\[\forall i=1,\ldots,k\colon  p_x^2(i)=p_z^2(i)=i,\]
then the dessin representing $\beta$ appears in Figure \ref{fig:cyl-lemma} (where $D_n,\, E_n,\, F_n$ and $G_n$ are dessins with $n$ black vertices---so we can even write $A=F_1,\,B=D_1,\,C=G_2$):
\begin{figure}[h]
\begin{center}
 \includegraphics[scale=0.8]{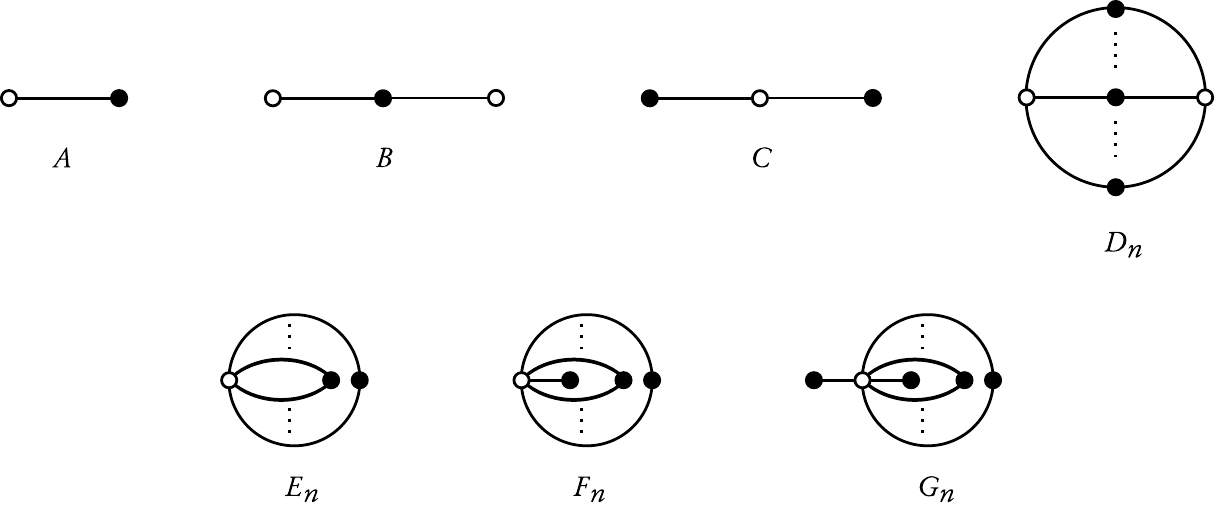}
 \caption{}\label{fig:cyl-lemma}
\end{center}
\end{figure}
Furthermore, under these conditions $\beta$ is defined over $\QQ$.
\end{hilfssatz}
%

\begin{proof}
The idea of the proof is to take cells of length 1 and 2 (this is the condition $p_z^2(i)=i$), bounded by edges, and glue them (preserving the orientation) around a white vertex $w$ until the cycle around this vertex is finished, i.e.\ until there are no more un-glued edges ending in $w$. The building blocks for this procedure are shown as 1a and 2a in the figure below. Note that it is not a priori clear that we end up with a closed surface in this process, but the proof will show that this is inevitable.

As a first step we consider in which ways we can identify edges or vertices within a cell of length 1 or 2, meeting the requirements that no black vertex shall have a valence $>2$ (this is the condition $p_x^2(i)=i$) and that the gluing respects the colouring of the vertices. The reader is invited to check that Figure \ref{fig:cyl-lemma-cells} lists all the possible ways of doing this.
\begin{figure}[h]
\begin{center}
 \includegraphics[scale=0.8]{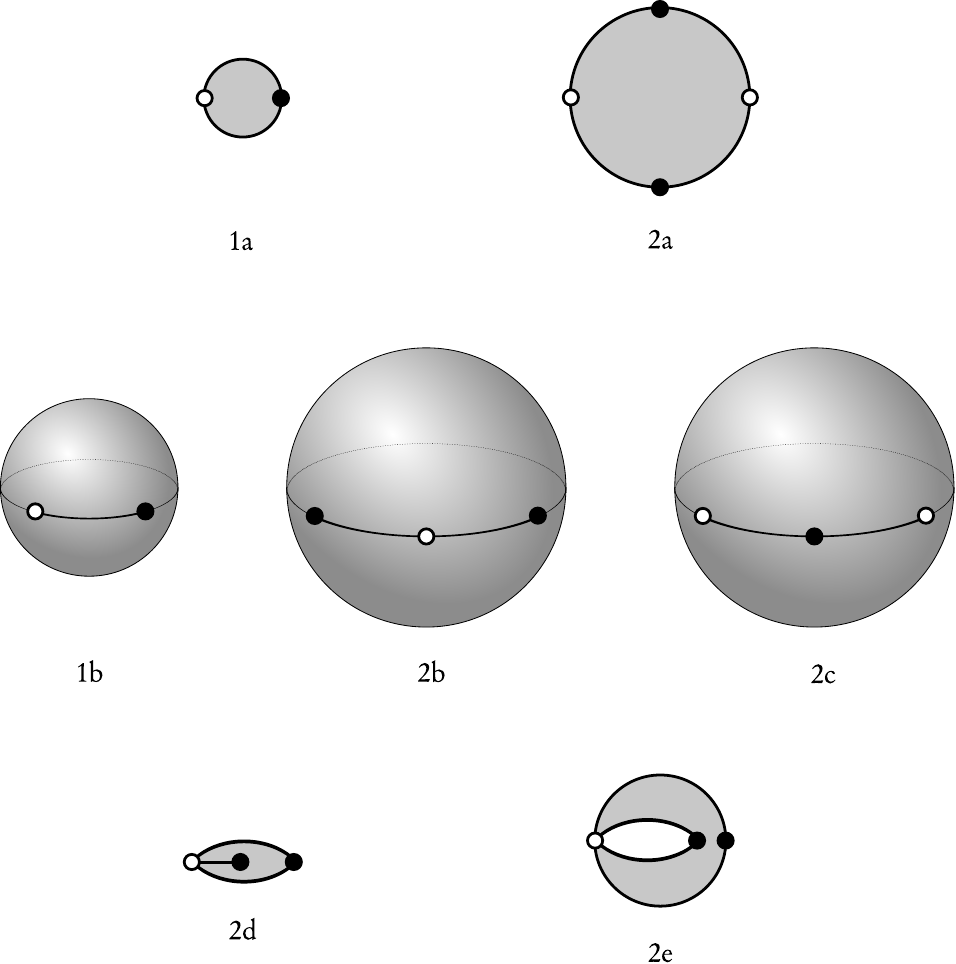}
\caption{}\label{fig:cyl-lemma-cells}
\end{center}
\end{figure}
Now we want to see in which ways we can glue these seven building blocks around the white vertex $w$. In the cases 1b, 2b and 2c we already have a closed surface, so we cannot attach any more cells, and end up with the dessins A, B and C. If we start by attaching the cell 2a to $w$, we can only attach other cells of that type to $w$. Note that we cannot attach anything else to the other white vertex, as this would force us to attach a cell of length $>2$ to $w$ in order to finish the cycle. So, in this case, we get the dessins of series D. Now consider the cell 2e. It has 4 bounding edges. To each of its two sides, we can attach another copy of 2e, leaving the number of bounding edges of the resulting object invariant, or a copy of 1a or 2d, both choices diminishing the number of boundary edges by $2$. Having a chain of 2e's, we cannot glue their two boundary components, as this would not yield a topological surface but rather a sphere with two points identified. So to close the cycle around the white vertex $w$, we have to attach on either side either 1a or 2d. This way we get the series E, F and G. If, on the other hand, we start with a cell of types 1a or 2d, we note that we can only attach cells of type 1a, 2d or 2e, so we get nothing new.

It might be surprising at first that constructing one cycle of $p_y$ with the given properties already determines the whole dessin. Again, the reader is invited to list the possibilities we seem to have forgotten, and check that they are in fact already in our list.

Now, clearly all of the dessins of type $\text{A},\ldots,\text{E}$ are determined by their cycle structure and hence defined over $\QQ$.
\end{proof}

With this lemma, we are now able to calculate the decomposition of an M-Origami into maximal cylinders. For a maximal cylinder of width $w$ and height $h$, we write that it is of type $(w,h)$.

\begin{satz}\label{satz:cyl-decomposition}
Let $O_\beta$ be an M-Origami associated to a dessin $\beta$ which is given by a pair of permutations $(p_x,\,p_y)$. Then we have:
\begin{aenum}
\item If $(p_x,\,p_y)$ does not define one of the dessins listed in Lemma \ref{hilfssatz:cyl-lemma}, then, in the Strebel direction $\left(\begin{smallmatrix}1\\0\end{smallmatrix}\right)$, $O_\beta$ has:
    \begin{itemize}
        \item for each fixed point of $p_y$ one maximal cylinder of type $(2,2)$,
        \item for each cycle of length $2$ of $p_y$ one maximal cylinder of type $(4,2)$,
        \item for each cycle of length $l>2$ of $p_y$ two maximal cylinders of type $(2l,1)$.
    \end{itemize}
\item In particular, we have in this case:
\[\#\,\text{max. horizontal cylinders}=2\cdot\#\,\text{cycles in }p_y-\#\,\text{fixed points of }p_y^2.\]
\item We get the maximal cylinders of $O_\beta$ in the Strebel directions $\left(\begin{smallmatrix}0\\1\end{smallmatrix}\right)$ and $\left(\begin{smallmatrix}1\\1\end{smallmatrix}\right)$ if we replace in a) the pair $(p_x,\,p_y)$ by $(p_y,\,p_x)$ and $(p_y,\,p_z)$, respectively.
\end{aenum}
\end{satz}

\begin{proof}
\begin{aenum}
\item First we note that for each cycle $c$ of $p_y$ of length $l$, we get two horizontal cylinders of length $2l$, one consisting of squares labelled $(1,j)$ and $(2,j)$, and one consisting of squares labelled $(3,j)$ and $(4,j)$---the latter rather belonging to the corresponding inverse cycle in $p_y^{-1}$. They are maximal iff there lie ramification points on both their boundaries (except in the trivial case where $g(O_\beta)=1$). Since the map $\pi$ from the construction of an M-Origami is always unramified over $\lambda$ (see Theorem \ref{satz:moncalc}), there are ramification points on the boundary in the middle of such a pair of cylinders iff the corresponding cycle is not self inverse, i.e.\ $l>2$. On the other two boundary components, there are ramification points iff not for every entry $j$ appearing in $c$, we have $p_x^2(j)=p_z^2(j)=1$. This is due to Lemma \ref{hilfssatz:mon-around-weierstrass}, and it is exactly the condition in Lemma \ref{hilfssatz:cyl-lemma}.
\item is a direct consequence of a).
\item By Proposition \ref{prop:ori-strebel-cylinders} d) we know that the cylinders in vertical and diagonal direction are the horizontal ones of $S^{-1}\cdot O_\beta$ and $(TS)^{-1}\cdot O_\beta$. As we know that $S\equiv S^{-1}$ and $(TS)^{-1}\equiv ST\:(\text{mod }\Gamma(2))$ we read off the claim from the table in Corollary \ref{kor:sl2-action}.
\end{aenum}
\end{proof}

For the sake of completeness, let us list the maximal horizontal cylinders of the M-Origamis associated to the exceptional dessins of Lemma \ref{hilfssatz:cyl-lemma}. We omit the easy calculations.
\begin{bem}
    \begin{aenum}
        \item The M-Origami coming from $D_n$ has two maximal horizontal cylinders of type $(2n,2)$ for $n\geq 3$ (and else one of type $(2n,4)$).
        \item The M-Origami coming from $E_n$ has one maximal horizontal cylinder of type $(4n,2)$.
        \item The M-Origami coming from $F_n$ has one maximal horizontal cylinder of type $(4n-2,2)$.
        \item The M-Origami coming from $G_n$ has one maximal horizontal cylinder of type $(4n-4,2)$.
    \end{aenum}
\end{bem}

\subsection{Möller's theorem and variations}

We are now able to reprove Möller's main result in \cite{mm1} in an almost purely topological way.  This will allow us to construct  explicit examples of origamis such that $\absGal$ acts non-trivially on the corresponding origami curves, which was not obviously possible in the original setting. 

Let us reformulate Theorem 5.4 from \cite{mm1}:
\begin{satz}[M.\ Möller]\label{msatz}
\begin{aenum}
\item Let $\sigma\in\absGal$ be an element of the absolute Galois group, and $\beta$ be a Belyi morphism corresponding to a clean tree, i.e.\ a dessin of genus $0$, totally ramified over $\infty$, such that all the preimages of $1$ are ramification points of order precisely $2$, and assume that $\beta$ is not fixed by $\sigma$. Then we also have for the origami curve $C(O_\beta)$ of the associated M-Origami: $C(O_\beta)\neq C(O_\beta)^\sigma$ (as subvarieties of $M_{g,[n]}$).
\item In particular, the action of $\absGal$ on the set of all origami curves is faithful.
\end{aenum}
\end{satz}
We will gather some lemmas which will enable us to reprove the above theorem within the scope of this work and to prove similar statements to part a) for other classes of dessins. Let us begin with the following simple

\begin{hilfssatz}\label{hilfssatz:beta-to-pi} Let $\beta\ncong\beta'$ be two dessins, defined by $(p_x,\,p_y)$ and $(p'_x,\,p'_y)$, respectively. Then, for their pullbacks $\pi, \pi'$ as in Definition \ref{defi:m-ori} we also have $\pi\ncong\pi'$.
\end{hilfssatz}
\begin{proof}
Let $\deg(\beta)=\deg(\beta')=d$ (if their degrees differ, the statement is trivially true). Now assume $\pi\cong\pi'$. This would imply the existence of a permutation $\alpha\in S_d$ such that $c_\alpha\circ m_\pi=m_{\pi'}$, where $c_\alpha$ is the conjugation by $\alpha$. By the proof of Theorem~\ref{satz:moncalc}, we have $m_\pi(d)=p_x,\,m_\pi(a)=p_y$, so in particular $(c_\alpha(p_x),\,c_\alpha(p_y))=(p_x',\,p_y')$, which contradicts the assertion that $\beta\ncong\beta'$.
\end{proof}
Next, let us check what happens after postcomposing $[2]$, the multiplication by $2$ on the elliptic curve. But we first need the following

\begin{hilfssatz}\label{hilfssatz:deck-existence}
Let $\pi\colon X\to E,\,\pi'\colon X'\to E$ be two coverings. If we have $[2]\circ\pi\cong[2]\circ\pi'$, then there is a deck transformation $\varphi\in \text{Deck}([2])$ such that $\varphi\circ\pi\cong \pi'$.
\end{hilfssatz}
\begin{proof} 
Clearly, $E$ is Hausdorff, and $[2]$ is a normal covering, so we can apply Lemma \ref{hilfssatz:decktrafo}.
\end{proof}
\begin{prop}\label{prop:distinct-origamis}
Assume we have a dessin $\beta$ given by $(p_x,\,p_y)$, and a Galois automorphism $\sigma\in\absGal$ such that $\beta\ncong\beta^\sigma$. If furthermore the 4-tuple $(p_x^2,\,p_y^2,\,p_z^2,\,1)\in (S_d)^4$ contains one permutation with cycle structure distinct from the others, then we have $O_\beta\ncong (O_{\beta})^\sigma$.
\end{prop}
\begin{proof}
First of all, what is $(O_\beta)^\sigma$? We chose $E$ to be defined over $\QQ$, so $E\cong E^\sigma$, and $[2]$ is also defined over $\QQ$, so $[2]=[2]^\sigma$, and so $([2]\circ\pi)^\sigma=[2]\circ\pi^\sigma$ which means by definition that $(O_\beta)^\sigma=O_{\beta^\sigma}$.

Assume now $O_\beta\cong O_{\beta^\sigma}$. So by the above lemma, there is a deck transformation $\varphi\in\Deck ([2])$ such that $\varphi\circ\pi \cong \pi^\sigma$. The deck transformation group here acts by translations, so in particular without fixed points. Note that by Lemma \ref{hilfssatz:mon-around-weierstrass}, the tuple $(p_x^2,\,p_y^2,\,p_z^2,\,1)$ describes the ramification of $\pi$ at the Weierstrass points, and the ramification behaviour of $\pi^\sigma$ is the same. So, imposing the condition that one entry in this tuple shall have a cycle structure distinct from the others, it follows that $\varphi=\id$. So we have even $\pi\cong\pi^\sigma$, and so by Lemma \ref{hilfssatz:beta-to-pi} $\beta\cong\beta^\sigma$, which contradicts the assumption.
\end{proof}

Now, we have all the tools together to prove Theorem \ref{msatz}.

\begin{proof}[Proof of Theorem \ref{msatz}]
The second claim follows from the first, as the action of $\absGal$ is faithful on trees, and it stays faithful if we restrict to clean ones, as we noted in Theorem \ref{satz:dessinaction}.

So, choose a non-trivial Galois automorphism $\sigma$ and a clean tree $\beta$ of degree $d$ defined by $(p_x,\,p_y)$ such that $\beta\ncong\beta^\sigma$.

First we check the condition of Proposition \ref{prop:distinct-origamis} by showing that the cycle structure of $p_z^2$ is distinct from the others. $p_z^2$ consists of two cycles because purity implies even parity of $d$. As $\beta\ncong\beta^\sigma$, surely $d>2$, so $p_z^2\neq 1$, and so it is distinct from $p_y^2=1$. Because of the purity, the dessin $\beta$ has $\tfrac{d}{2}$ white vertices, and so $\tfrac{d}{2}+1$ black vertices, which is a lower bound for the number of cycles in $p_x^2$. Again, from $d>2$ we conclude that $p_x^2$ must consist of at least $3$ cycles and therefore cannot be conjugate to $p_z^2$.

We claim now that $O_\beta$ and $O_{\beta^\sigma}$ are not affinely equivalent. By Theorem \ref{satz:veechgroup} and Corollary \ref{kor:sl2-action} this amounts to checking that $\beta$ and $\beta^\sigma$ are not weakly isomorphic. As we will see later in an example, this cannot be assumed in general, but in this case $p_x,\,p_y,\,p_z$ consist of $\tfrac{d}{2}+1$ cycles, $\tfrac{d}{2}$ cycles and $1$ cycle, respectively, so any weak isomorphism would actually have to be an isomorphism, which we excluded.

So, by Proposition \ref{prop:aff-eq-origamis} , $C(O_\beta)\neq C(O_{\beta^\sigma})$ as embedded curves in the moduli space. But as we saw in the proof of Proposition \ref{prop:distinct-origamis}, the latter is equal to $C((O_\beta)^\sigma)$, which is in turn equal to $(C(O_\beta))^\sigma$ by Proposition \ref{prop:ori-curves-arithmetic} b). Altogether we found, for an arbitrary $\sigma\in\absGal$, an origami $O$ such that
\[C(O)\neq(C(O))^\sigma.\]
So indeed, the absolute Galois group acts faithfully on the set of origami curves.
\end{proof}
Inspecting our results that we used to prove Möller's theorem more closely, we see that we can actually use them to give a larger class of dessins $\beta$  for which we know that from $\beta^\sigma\ncong\beta$ follows $C(O_\beta)\neq (C(O_\beta))^\sigma$:
\begin{satz}\label{satz:treefilth} Let $\beta$ be a dessin such that $\beta^\sigma\ncong\beta$ for some $\sigma\in\absGal$.
\begin{aenum}
\item If $\beta$ is a tree or a filthy dessin, then $O_\beta\ncong O_\beta^\sigma$.
\item If furthermore $M(\beta)=M_\beta$ (in the sense of Definitions \ref{defi:fields-morph} and \ref{defi:dessinmoduli}), then we have $C(O_\beta)\neq (C(O_\beta))^\sigma$.
\end{aenum}
\end{satz}

\begin{proof}
For part a), remember that a dessin with monodromy given by $(p_x,\,p_y)$ is said to be filthy if it is not weakly isomorphic to a pre-clean one, i.e.\ $1\notin\{p_x^2,\,p_y^2,\,p_z^2\}$. So the condition of Proposition \ref{prop:distinct-origamis} is satisfied for the permutation $1$.

The case of $\beta$ being a tree is a little bit more tricky. We want to show that the cycle structure of $p_z^2$ is unique. $\beta$ is totally ramified over $\infty$, so $p_z^2$ has one cycle if $d=\deg(\beta)$ is odd and two if it is even. First assume it to be odd. If $p_x^2$ (and so $p_x$) also had only one cycle, then we had $\beta\colon  z\mapsto z^d$ which contradicts the assumption $\beta\ncong\beta^\sigma$. We repeat the argument for $p_y$. Clearly $d>1$, so $p_z^2\neq 1$.

Assume now $d\in 2\ZZ$, so $p_z^2$ has two cycles of length $\tfrac{d}{2}$. Assume $p_x^2$ to be conjugate to it, then $p_x$ has either one cycle of length $d$ or two of length $\tfrac{d}{2}$. We already discussed the first case, and in the latter one, there is, for every $d$, only one tree with that property:

\begin{center}
 \includegraphics[scale=0.8]{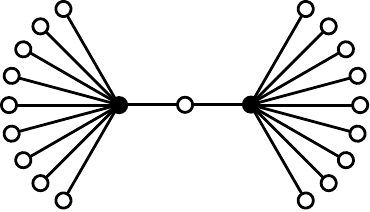}
\end{center}

So in particular it is defined over $\QQ$ (this can be seen by writing down its Belyi polynomial $\beta(z)=(-z^2+2z)^\frac{d}{2}$), which contradicts the hypothesis $\beta\ncong\beta^\sigma$. Again we repeat the argument for $p_y$, and as surely $d>2$ we have $p_z^2\neq 1$.

For part b), we have to check that $O_\beta$ and $O_\beta^\sigma$ are not affinely equivalent. Again by Theorem \ref{satz:veechgroup} b) this is equivalent to $\beta$ and $\beta^\sigma$ not being weakly isomorphic. But indeed, the conditon $M(\beta)=M_\beta$ implies that whenever $\beta$ and $\beta^\sigma$ are weakly isomorphic for some $\sigma\in\absGal$, we have actually $\beta\cong\beta^\sigma$.

\end{proof}

Note that Example \ref{bsp:smaller-curve-moduli} shows that the condition $M(\beta)=M_\beta$ is actually necessary: the Galois orbit of dessins considered in that example contains a pair of filthy trees which are weakly isomorphic. The corresponding M-Origamis are thus distinct by Theorem \ref{satz:treefilth} b), but in fact affinely equivalent.

From Corollary \ref{kor:finite} we know that in order to act faithfully on the Teichmüller curves of M-Origamis, the absolute Galois group must act non-trivially on the Teichmüller curves of M-Origamis of genus $g$ for infinitely many $g$. In fact, one can show that it acts non-trivially for each $g\geq 4$. This is done in \cite[Proposition 4.22]{diss}.

\section{Examples}

We will now illustrate the results of the previous section by some examples. First, we will construct two non-trivial Galois orbits of origami curves. To do so, we take two orbits of dessins from \cite{cat} and feed them into the M-Origami-machine. Then, we will amend them with two examples showing that every congruence subgroup in $\SL_2(\ZZ)$ of level $2$ actually appears as the Veech group of an M-Origami.
\begin{bsp}
Consider the following Galois orbit of dessins:
\vspace{5mm}

\begin{center}
 \includegraphics[scale=0.8]{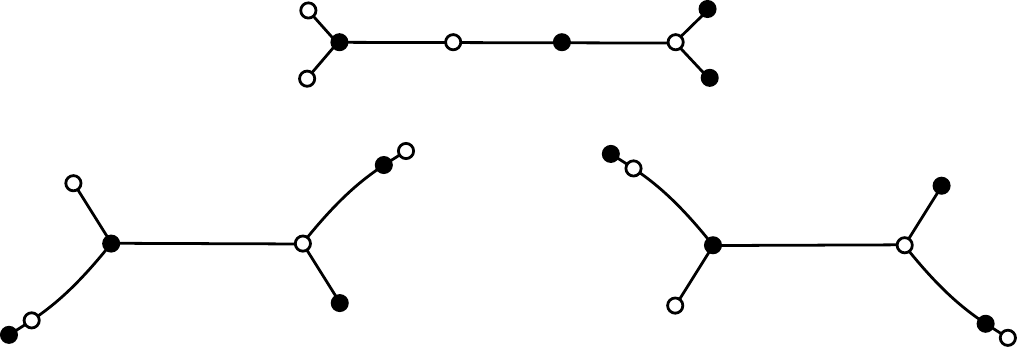}
\end{center}
\vspace{5mm}

The dessins are given by Belyi polynomials of the form
\[\beta(z)=z^3(z-a)^2\left(z^2+\left(2a-\frac{7}{2}\right)z+\frac{8}{5}a^2-\frac{28}{5}a+\frac{21}{5}\right),\]
where $a$ runs through the three complex roots of the polynomial
\[24 a^3-84 a^2 + 98 a -35,\]
one of which is real and the other two of which are complex conjugate.

We can number the edges in such a way that for all three we get $p_x=(1\,2\,3)(4\,5)$ and, from top to bottom right in the above picture, we get
\[p_y^{(1)}=(3\,4)(5\,6\,7),\,p_y^{(2)}=(2\,7)(3\,6\,4),\,p_y^{(3)}=(1\,7)(3\,4\,6).\]
First, we write down the monodromy of the corresponding M-Origamis, using Theorem \ref{satz:moncalc}: For all of them, we can choose $p_B$ to be
\begin{align*}
p_B &= (1\,3\,9\,11\,5\,7)(2\,4\,6\,8\,10\,12)(13\,15\,17\,19)(14\,16\,18\,20)\\
    &\quad\:\:(21\,23)(22\,24)(25\,27)(26\,28)
\end{align*}
and further we calculate
\begin{align*}
p_A^{(1)}&=(1\,2)(3\,4)(5\,6)(7\,8)(9\,10\,13\,14)(11\,12\,15\,16)\\
         &\quad\:\:(17\,18\,21\,22\,25\,26)(19\,20\,27\,28\,23\,24),\\
p_A^{(2)}&=(1\,2)(3\,4)(5\,6\,25\,26)(7\,8\,27\,28)(9\,10\,21\,22\,13\,14)\\
         &\quad\:\:(11\,12\,15\,16\,23\,24)(17\,18)(19\,20),\\
p_A^{(3)}&=(1\,2\,25\,26)(3\,4\,27\,28)(5\,6)(7\,8)(9\,10\,13\,14\,21\,22)\\
         &\quad\:\:(11\,12\,23\,24\,15\,16)(17\,18)(19\,20).
\end{align*}

Now let us draw the origamis—numbers in small print shall indicate the gluing. We can do that in a way that exhibits the mirror symmetry of the first one (which shows that the origami, and thus its curve, is defined over $\RR$), and the fact that the other two are mirror images of each other, i.e.\ they are interchanged by the complex conjugation.

\vspace{5mm}

\begin{center}
\begin{minipage}{10cm}
\begin{xy}
<0.6cm,0cm>:
(1,-1)*{\OriSquare{1}{}{}{2}{7}};
(2,-1)*{\OriSquare{2}{1}{}{}{}};
(1,0)*{\OriSquare{3}{}{9}{4}{}};
(2,0)*{\OriSquare{4}{3}{6}{}{}};
(5,-1)*{\OriSquare{5}{}{}{6}{}};
(6,-1)*{\OriSquare{6}{5}{}{}{4}};
(5,0)*{\OriSquare{7}{}{1}{8}{}};
(6,0)*{\OriSquare{8}{7}{10}{}{}};
(5,-3)*{\OriSquare{9}{10}{}{}{3}};
(2,-3)*{\OriSquare{10}{}{}{9}{8}};
(5,-2)*{\OriSquare{11}{12}{}{}{}};
(2,-2)*{\OriSquare{12}{}{}{11}{}};
(3,-3)*{\OriSquare{13}{}{}{}{}};
(4,-3)*{\OriSquare{14}{}{}{}{}};
(3,-2)*{\OriSquare{15}{}{17}{}{}};
(4,-2)*{\OriSquare{16}{}{18}{}{}};
(3,-5)*{\OriSquare{17}{}{}{}{15}};
(4,-5)*{\OriSquare{18}{}{}{}{16}};
(3,-4)*{\OriSquare{19}{}{}{24}{}};
(4,-4)*{\OriSquare{20}{27}{}{}{}};
(5,-5)*{\OriSquare{21}{}{23}{}{}};
(6,-5)*{\OriSquare{22}{25}{24}{}{}};
(5,-6)*{\OriSquare{23}{}{}{28}{21}};
(6,-6)*{\OriSquare{24}{19}{}{}{22}};
(1,-5)*{\OriSquare{25}{}{27}{22}{}};
(2,-5)*{\OriSquare{26}{}{28}{}{}};
(1,-6)*{\OriSquare{27}{}{}{20}{25}};
(2,-6)*{\OriSquare{28}{23}{}{}{26}};
\end{xy}
\end{minipage}
\end{center}
\vspace{5mm}

\begin{center}
\begin{minipage}{10cm}
\begin{xy}
<0.6cm,0cm>:
(1,-3)*{\OriSquare{1}{}{}{2}{7}};
(2,-3)*{\OriSquare{2}{1}{}{}{12}};
(1,-2)*{\OriSquare{3}{}{9}{4}{}};
(2,-2)*{\OriSquare{4}{3}{}{}{}};
(5,-1)*{\OriSquare{5}{6}{}{}{}};
(2,-1)*{\OriSquare{6}{}{}{5}{}};
(5,0)*{\OriSquare{7}{8}{1}{}{}};
(2,0)*{\OriSquare{8}{}{10}{7}{}};
(5,-3)*{\OriSquare{9}{}{}{14}{3}};
(6,-3)*{\OriSquare{10}{}{}{}{8}};
(5,-2)*{\OriSquare{11}{}{}{24}{}};
(6,-2)*{\OriSquare{12}{15}{2}{}{}};
(9,-3)*{\OriSquare{13}{}{}{}{19}};
(10,-3)*{\OriSquare{14}{9}{}{}{20}};
(9,-2)*{\OriSquare{15}{}{}{12}{}};
(10,-2)*{\OriSquare{16}{23}{}{}{}};
(9,-1)*{\OriSquare{17}{}{}{18}{}};
(10,-1)*{\OriSquare{18}{17}{}{}{}};
(9,-0)*{\OriSquare{19}{}{13}{20}{}};
(10,-0)*{\OriSquare{20}{19}{14}{}{}};
(7,-3)*{\OriSquare{21}{}{23}{}{}};
(8,-3)*{\OriSquare{22}{}{24}{}{}};
(7,-4)*{\OriSquare{23}{}{}{16}{21}};
(8,-4)*{\OriSquare{24}{11}{}{}{22}};
(3,-1)*{\OriSquare{25}{}{}{}{27}};
(4,-1)*{\OriSquare{26}{}{}{}{28}};
(3,0)*{\OriSquare{27}{}{25}{}{}};
(4,0)*{\OriSquare{28}{}{26}{}{}};
\end{xy}
\end{minipage}
\vspace{5mm}

\begin{minipage}{10cm}
\begin{xy}
<0.6cm,0cm>:
(4,-5)*{\OriSquare{1}{2}{}{}{7}};
(1,-5)*{\OriSquare{2}{}{}{1}{12}};
(4,-4)*{\OriSquare{3}{4}{}{}{}};
(1,-4)*{\OriSquare{4}{}{}{3}{}};
(0,-3)*{\OriSquare{5}{}{}{6}{11}};
(1,-3)*{\OriSquare{6}{5}{}{}{}};
(0,-2)*{\OriSquare{7}{}{1}{8}{}};
(1,-2)*{\OriSquare{8}{7}{10}{}{}};
(4,-3)*{\OriSquare{9}{}{}{22}{}};
(5,-3)*{\OriSquare{10}{13}{}{}{8}};
(4,-2)*{\OriSquare{11}{}{5}{16}{}};
(5,-2)*{\OriSquare{12}{}{2}{}{}};
(8,-3)*{\OriSquare{13}{}{}{10}{}};
(9,-3)*{\OriSquare{14}{21}{}{}{}};
(8,-2)*{\OriSquare{15}{}{17}{}{}};
(9,-2)*{\OriSquare{16}{11}{18}{}{}};
(8,-5)*{\OriSquare{17}{}{}{18}{15}};
(9,-5)*{\OriSquare{18}{17}{}{}{16}};
(8,-4)*{\OriSquare{19}{}{}{20}{}};
(9,-4)*{\OriSquare{20}{19}{}{}{}};
(6,-1)*{\OriSquare{21}{}{23}{14}{}};
(7,-1)*{\OriSquare{22}{9}{24}{}{}};
(6,-2)*{\OriSquare{23}{}{}{}{21}};
(7,-2)*{\OriSquare{24}{}{}{}{22}};
(2,-5)*{\OriSquare{25}{}{}{}{27}};
(3,-5)*{\OriSquare{26}{}{}{}{28}};
(2,-4)*{\OriSquare{27}{}{25}{}{}};
(3,-4)*{\OriSquare{28}{}{26}{}{}};
\end{xy}
\end{minipage}
\end{center}
\vspace{5mm}

Using Proposition \ref{prop:genus} and Remark \ref{bem:m-ori-punctures}, we see that all of these three M-Origamis of degree $28$ have genus $6$ and $18$ punctures. The interesting fact is that the three dessins admit a weak automorphism lying over $z\mapsto 1-z$, i.e.\ they stay the same after exchanging white and black vertices. By Corollary \ref{kor:sl2-action}, this means that $S$ is contained in each of their Veech groups. $T$ is contained in neither of them, because by the same corollary this cannot happen for nontrivial trees, and so all three have the Veech group generated by $\Gamma(2)$ and $S$, so their Teichmüller curves have genus $0$ with $2$ cusps, and no two of these Teichmüller curves coincide.
\end{bsp}
\begin{bsp}\label{bsp:smaller-curve-moduli}
There is a second non-trivial Galois orbit of genus $0$ dessins of degree $7$:
\vspace{5mm}

\begin{center}
 \includegraphics[scale=0.8]{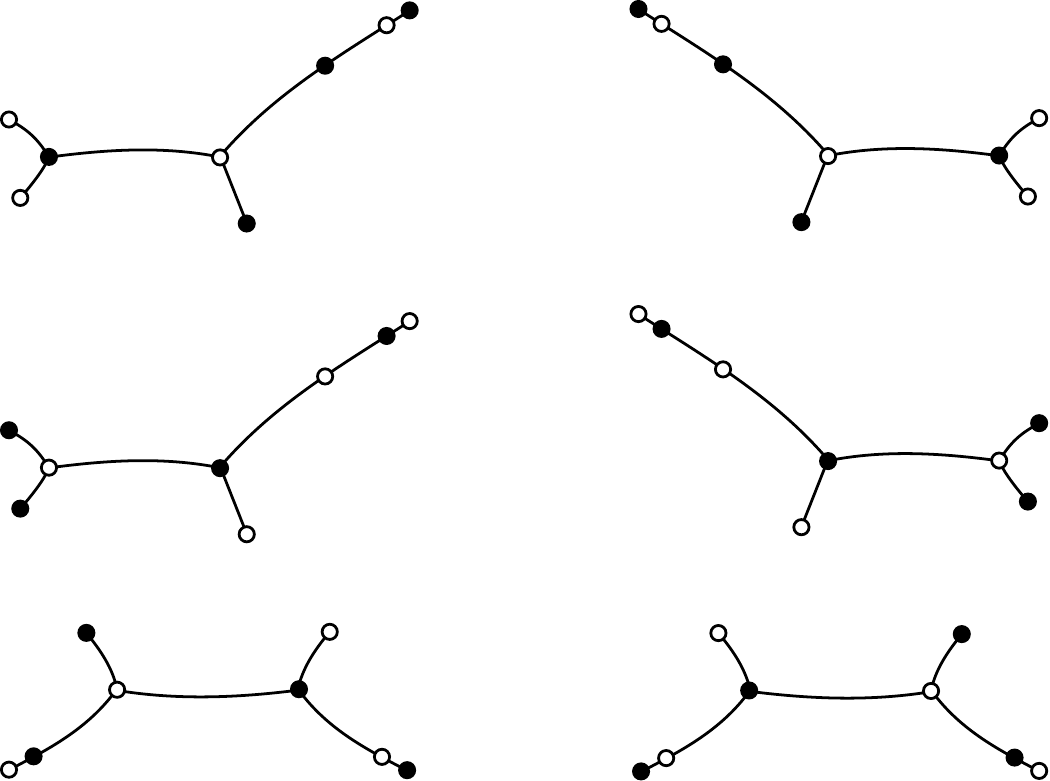}
\end{center}
\vspace{5mm}

From the catalogue we learn that their Belyi polynomials are of the form
\begin{multline*}
\beta(z)=z^3(z-a)^2\bigg(z^2+\Big(\frac{4}{3}a^5-\frac{34}{15}a^4-\frac{26}{15}a^3+\frac{7}{5}a^2+\frac{20}{3}a-\frac{28}{5}\Big)z\\
-\frac{8}{15}a^5-\frac{32}{75}a^4+\frac{172}{75}a^3+\frac{148}{75}a^2-\frac{14}{5}a - \frac{287}{75}\bigg),
\end{multline*}
where $a$ runs through the six complex roots of the following polynomial:
\[20a^6-84a^5+84a^4+56a^3-294a+245\]
As they have the same cycle structure as the three in the Galois orbit discussed in the previous example, the resulting M-Origamis are of course combinatorially equivalent to them: They also are of degree $28$ and genus $6$, and they have $18$ punctures. (We omit to draw them here.) But something is different here: the weak automorphism group of each of these dessins is trivial. To see this, note that an element of the Group $W$ stabilising a tree (that is not the dessin of $z\mapsto z^n$, or $z\mapsto (1-z)^n$) has to fix $\infty$, so it is either the identity or the element $s$ which acts by exchanging the white and black vertices. Indeed, none of these six dessins keep fixed under $s$ (but interestingly the whole orbit does). As these dessins are filthy, we conclude by Theorem \ref{satz:veechgroup} c) that all six corresponding M-Origamis have Veech group $\Gamma(2)$. Furthermore, the two columns of the picture are complex conjugate to each other, and we get the first row from the second by applying $s$---which is also the case for the two dessins from the last row. So what is the situation here? We get three Teichmüller curves, the first two containing the origamis associated to the two upper left and upper right dessins, respectively. They are interchanged by the complex conjugation. The third curve contains the two other origamis, associated to the bottom row, so this Teichmüller curve is stabilised under the action of the complex conjugation, and hence defined over $\RR$.
\end{bsp}

Let us turn away from the Galois action now and concentrate on the possible Veech groups of M-Origamis. As we have learned in Proposition \ref{prop:equivariant-action}, it is closely connected tho the weak automorphism group $W_\beta$ of the underlying dessin. The remaining possibilities (up to conjugation) not discussed in the two previous examples are $W_\beta=W\cong S_3$ and $W_\beta=\langle st\rangle\cong \ZZ/3\ZZ$. We will indeed present non-trivial examples of such dessins. It would be interesting to have a classification of dessins with these weak automorphism groups.

\begin{bsp}\label{bsp:stefan}
Here, we construct an infinite series of M-Origamis with Veech group $\SL_2(\ZZ)$. It is quite noteworthy that only the two simplest origamis in this series are characteristic. An origami $O=(f\colon X^*\to E^*)$ is called \textit{characteristic} if $f_*(\pi_1(X^*))\leq\pi_1(E^*)\cong F_2$ is a characteristic subgroup. For a detailed account of characteristic origamis see \cite{he}.

We begin with the following series of dessins that S.\ Kühnlein came up with:

Let $n\geq 2$. Consider on the set $(\ZZ/n\ZZ)^2$ the following two maps:
\[p_x\colon (k,\,l)\mapsto (k+1,\,l),\:\:p_y\colon (k,\,l)\mapsto(k,\,l+1).\]
They are clearly bijective, so we can regard them as elements of $S_{n^2}$. Note that they commute. Also, as $p_x^kp_y^l(0,\,0)=(k,\,l)$, they generate a transitive subgroup of $S_{n^2}$ and so $(p_x,\,p_y)$ defines a dessin $K_n$ of degree $n^2$. We calculate its monodromy around $\infty$ as $p_z=p_x^{-1}p_y^{-1}\colon (k,\,l)\mapsto (k-1,\,l-1)$ and define
\[c\colon (k,l)\mapsto (l,\,k),\:\:d\colon (k,\,l)\mapsto(-k,\,l-k).\]
As $c^2=d^2=\id$, both $c$ and $d$ are also bijective and thus elements of $S_{n^2}$. Furthermore we easily verify that
\[c^{-1}p_xc=p_y,\,c^{-1}p_yc=p_x,\,d^{-1}p_xd=p_z,\,d^{-1}p_yd=p_y,\]
so we find that $s\cdot K_n\cong t\cdot K_n\cong K_n$ and thus by Definition and Remark \ref{defbem:weak-action}, $W_{K_n}=W$. Before we go on, we calculate the genus of $K_n$. The permutation $p_x$ consists of $n$ cycles (of length $n$), and so do $p_y$ and $p_z$. As its degree is $n^2$, by the Euler formula we get $2-2g(K_n)=2n-n^2+n$ and thus
\[g(K_n)=\frac{n^2-3n+2}{2}.\]
Now, look at the associated M-Origami $O_{K_n}$ of degree $4n^2$. By the considerations above, we know that it has Veech group $\SL_2(\ZZ)$. But since for $n\geq 3$, we have $1\notin\{p_x^2,\,p_y^2,p_z^2\}$, the group of translations cannot act transitively on the squares of $O_{K_n}$ which therefore is not a normal and specifically not a characteristic origami. According to a remark in \cite{he}, it seems as if not many examples are known for non-characteristic origamis with full Veech group $\SL_2(\ZZ)$. Let us close the example by calculating the genus of $O_{K_n}$. We use the formula from Proposition \ref{prop:genus} and therefore we have to count the cycles of even length in $p_x,\,p_y$ and $p_z$. We have
\[g_0=g_1=g_\infty=\begin{cases}n,& n\in 2\ZZ\\0,& n\in 2\ZZ+1\end{cases}\]
and therefore
\[g(O_{K_n})=g(K_n)+n^2-\frac{1}{2}(g_0+g_1+g_\infty)=
\begin{cases}
\frac{3n^2-6n+2}{2},&n\in2\ZZ\\
\frac{3n^2-3n+2}{2},&n\in2\ZZ+1
\end{cases}.\]
\end{bsp}

\begin{bsp}
Consider the following genus $1$ dessin $\beta$:
\vspace{5mm}

\begin{center}
 \includegraphics[scale=0.8]{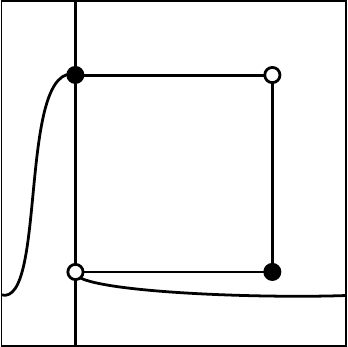}
\end{center}
\vspace{5mm}

Its monodromy is given by $p_x=(1\,3\,4\,5)(2\,6)$ and $p_y=(2\,1\,4\,5)(3\,6)$. We set $c\coloneqq (1\,2\,3)$ and easily check
\[cp_xc^{-1}=p_y,\:cp_yc^{-1}=p_z=p_x^{-1}p_y^{-1},\:cp_zc^{-1}=p_x.\]
It can be easily verified that no permutation in the centraliser of $p_z$ exchanges $p_x$ and $p_y$ by conjugation, so the weak isomporphism group $W_\beta$ is indeed generated by the element $st$ and of order $3$. So we have $\langle\Gamma(2),ST\rangle\subseteq\Gamma(O_\beta)$. Since $\beta$ is filthy, we can conclude as in the proof of Theorem \ref{satz:veechgroup} c) that $O_\beta\ncong S\cdot O_\beta$, so indeed we have $\Gamma(O_\beta)=\langle\Gamma(2),ST\rangle.$
\end{bsp}

Considering all four examples (plus the $\SL_2(\ZZ)$-orbit of the second one) we thus get:
\begin{prop}
Every congruence subgroup of $\SL_2(\ZZ)$ of level $2$ appears as the Veech group of an M-Origami.
\end{prop}
\printbibliography
\end{document}